%% file: main.tex
\documentclass[12pt]{amsart}
\usepackage[utf8]{inputenc}

\usepackage{geometry}
\usepackage{setspace}
\usepackage{graphicx}
\usepackage{amssymb}
\usepackage{mathrsfs}
\usepackage{amsthm}
\usepackage{amsmath}
\usepackage{color}
\usepackage{caption}
\usepackage{comment}
\usepackage{subcaption}
\usepackage{mathtools}
\usepackage{enumerate}
\usepackage{algorithm,algorithmic}

\usepackage{hyperref}

\theoremstyle{plain} % Statements in italic
\newtheorem{theorem}{Theorem}[section]

\newtheorem{corollary}[theorem]{Corollary}

\newtheorem{lemma}[theorem]{Lemma}

\theoremstyle{definition} % Statements in non-italic
\newtheorem{example}[theorem]{Example}
\newtheorem{assumption}[theorem]{Assumption}

\theoremstyle{remark} % Statements in non-italic
\newtheorem{remark}[theorem]{Remark}

\numberwithin{equation}{section} % Equations numbered after section

\makeatletter
\newcommand{\myitem}[1]{%
\item[#1]\protected@edef\@currentlabel{#1}%
}
\makeatother

\title[Finite element approximation of SPDEs with Whittle--Matérn noise]{Finite element approximation of parabolic SPDEs with Whittle--Matérn noise}

\author[Ø.~S.~Auestad]{Øyvind S. Auestad} 
\address[Øyvind S. Auestad]{\newline Department of Mathematical Sciences Norwegian University of Science and Technology, \newline 7034 Trondheim, Norway.} \email[(corresponding author)]{oyvinau@ntnu.no}

\author[G.-A.~Fuglstad]{Geir-Arne Fuglstad}
\address[Geir-Arne Fuglstad]{\newline Department of Mathematical Sciences Norwegian University of Science and Technology \newline NO--7491 Trondheim, Norway.} \email{geir-arne.fuglstad@ntnu.no}

\author[E.~R.~Jakobsen]{Espen R. Jakobsen}
\address[Espen R. Jakobsen]{\newline Department of Mathematical Sciences Norwegian University of Science and Technology \newline NO--7491 Trondheim, Norway.} \email{espen.jakobsen@ntnu.no}

\author[A.~Lang]{Annika Lang} \address[Annika Lang]{\newline Department of Mathematical Sciences Chalmers University of Technology and University of Gothenburg \newline S--412 96 G\"oteborg, Sweden.} \email{annika.lang@chalmers.se}

\begin{document}

\begin{abstract}
    We propose and analyse a new type of fully discrete finite element approximation of a class of linear stochastic parabolic evolution equations with additive noise. Our discretization differs from previous ones in that we use a finite element approximation of the noise, as opposed to an $L^2$ projection. This approximation is tailored for equations where the noise has covariance operator defined in terms of (negative powers of) elliptic operators, like Whittle--Matérn random fields. Strong convergence rates up to order~$2$ in space and~$1$ in time are shown and verified by numerical experiments in dimension $1$ and~$2$.
\end{abstract}

\date{\today}

\maketitle

\section{Introduction}

We consider linear stochastic evolution equations of the form,
\begin{align}\label{eq:model-in-intro}
    du = -A_1 u \, dt + A_2^{-\gamma} \, dW, \quad u(0) = \xi,
\end{align}
where $W$ is a cylindrical Wiener process on $L^2(\mathcal{D})$, $\mathcal{D} \subseteq \mathbb{R}^d$ is a bounded domain, $\gamma > d/4-1/2$ is a parameter, $\xi$ a random initial condition, and $A_j, \ j = 1, 2$ second order elliptic operators of the form, 
\begin{align*}
    A_j = \alpha_j - \nabla \cdot \mathcal{A}_j \nabla + b_j \cdot \nabla, \quad j = 1, 2,
\end{align*}
with bounded and measurable coefficients---for precise conditions, see Assumption \ref{assumption:model}. The model \eqref{eq:model-in-intro} has additive colored noise of Whittle--Matérn type, and is of particular interest in spatial statistics. In this paper we propose and analyse and a new type of fully discrete finite element approximation for this class of equations, which uses a combination of a finite element approximation and a quadrature to approximate the driving noise $A_2^{-\gamma} W$.

Strong convergence of fully discrete numerical approximations to mild solutions of linear and semilinear stochastic partial differential equations (SPDEs) has been studied for more than $25$~years starting with the first works by Gy\"ongy~\cite{1999-gyongy} and Shardlow~\cite{1999-shardlow}. While the convergence order in time of Euler--Maruyama schemes for SPDEs driven by multiplicative noise is limited to~$1/2$ as shown, e.g., in~\cite{kruse-2014}, additive noise opens up for up to order~$1$ as shown, e.g., in~\cite{2004-yan, 2016-wang} and for exponential integrators in~\cite{2019-lord}. Spatial discretization is usually done by spectral methods, finite differences, or finite element methods, see the reviews in the monographs \cite{lord, kruse, jentzen} and references therein. Common to most numerical discretizations in the literature is that the driving noise enters the scheme via the $L^2$-orthogonal projection onto some finite dimensional space. This can either be a finite element space in the case of a finite element approximation, or the span of suitably chosen eigenvectors in the case of a spectral approximation. However, neither of these choices usually results in schemes which are computable in practice. Finite element approximations require methods of simulating the projected colored noise, while spectral approximations require knowledge of the eigenvectors and eigenvalues of the covariance operator of the noise---both of which are often unavailable. First approaches to construct a computable finite element approximation have been made by an additional spectral approximation of the noise in~\cite{2010-lindgren, 2012-barth} and based on covariance kernels and interpolation between meshes, e.g., via fast Fourier transform, in~\cite{2022-petersson}. In this work, we use a direct finite element based approximation of the noise, which can quickly be computed on the same finite element space as the solution to the equation.

More precisely, for a finite element space $V_h$ with $L^2$-orthogonal projection $\pi_h$, and finite element approximations $A_{j,h}$ of $A_j$, our approximation~$u_h$ is based on the semidiscrete approximation
\begin{align}\label{eq:mild-solution-semi-approx}
    d u_h = -A_{1,h} u_h \, dt + A_{2,h}^{-\gamma} \pi_h \, dW, \quad u_h(0) = \pi_h \xi.
\end{align}
In contrast, the common way of approximating the additive noise would be by replacing the last term in \eqref{eq:mild-solution-semi-approx} by $\pi_h A_2^{-\gamma} \, dW$. By using a quadrature to approximate an integral representation of the fractional operator $A_{2,h}^{-\gamma}$, and a backward Euler approximation in time, we arrive at our fully discrete approximation (see \eqref{eq:fully-discrete-scheme}). 

Our approximation is tailored for SPDEs where the noise has covariance operator determined by a negative fractional power of a differential operator. This choice of covariance operator is motivated by applications in spatial statistics, where the arguably most popular and widely applied model is the Matérn random field \cite{2012-stein, 2024-porcu}, which corresponds to the case $\mathcal{A}_2 = \tau I$, $b_2 = 0$, $\alpha_2, \tau > 0$, and $\mathcal{D} = \mathbb{R}^d$. The formulation of Matérn random fields as solutions to SPDEs, dating back to \cite{1954-whittle}, allows for a natural generalization of the Matérn random field to bounded domains and more general covariance operators. This class of random fields is defined as solutions $u$ of $A_2^{\gamma} u = \mathcal{W}$, where $\mathcal{W}$ is white noise on $L^2(\mathcal{D})$, and are often referred to as Whittle--Matérn random fields \cite{2022-lindgren}, motivating the name ``Whittle--Matérn noise" in this paper. Whittle--Matérn random fields and their numerical approximation have been a popular subject of research in mathematics and statistics in recent years (see, e.g., \cite{2018-bolin, 2020-cox, 2024-bolin}, \cite{2011-rue, 2022-lindgren, 2024-lindgren} and the references therein).

We may view \eqref{eq:model-in-intro} as a model for the time development of a density subject to diffusion and advection, in addition to random sources and fluxes, defined by Whittle--Matérn random fields. SPDEs of this form are known in statistics as stochastic advection-diffusion equations, and are used as statistical models for physical phenomena in climate and environmental sciences, such as precipitation, pollution and temperature \cite{2010-wikle,2011-cressie,2012-sigrist,2022-foss}, to name a few. The study of these models is an active area of research in spatio-temporal statistics \cite{2024-lindgren,2022-liu,2024-clarotto,2024-berild}, where a key objective is to create physically motivated statistical models that are interpretable and computationally feasible for statistical inference. From a statistical perspective, the interpretability and linearity of \eqref{eq:model-in-intro} are favorable, as standard statistical inference requires the computation of likelihoods, which is typically not available for non-Gaussian and non-linear models.

In the spatio-temporal statistics literature, the motivation for studying models of the form \eqref{eq:model-in-intro} are applications involving real data. In \cite{2024-clarotto}, a finite element method that is a special case of the method developed in this paper is applied in the case $\gamma = 1$ and $b_2 = 0$, and \cite{2024-lindgren} considers the case $b_1=b_2=0$ with a finite element approach in time combined with the same spatial noise discretization for $\gamma \in \mathbb{N}$, in addition to allowing $\gamma \in \mathbb{N} - 1/2$ through a least squares method. Common for \cite{2024-clarotto}, \cite{2024-lindgren}, other work in statistics using finite volume methods \cite{2024-berild}, and spectral approaches \cite{2022-liu,2015-sigrist}, is that they are missing a rigorous formulation of the requirements on the SPDE \eqref{eq:model-in-intro} and numerical method, proofs of rates of convergence, in addition to numerical experiments verifying convergence rates. Therefore, this paper closes the gap between some of the numerical discretizations of \eqref{eq:model-in-intro} used in practice in spatio-temporal statistics, and available convergence results in the numerical analysis literature. The extension from integer $\gamma$ to general $\gamma > d/4 - 1/2$ is important because $\gamma$ provides a way to control spatial regularity of the noise, and thus the solution. The ability to control spatial regularity is key in applications and one of the main motivations for using the Matérn model \cite{2012-stein}. In Figure \ref{ga:example3:fig:NonStat} realizations of \eqref{eq:model-in-intro} are shown for different choices of $\gamma$ (see Example \ref{ga:example3} for details).

\begin{figure}
    \centering
    \subcaptionbox{$\gamma = 0.25$}{
        \includegraphics[width = 16cm]{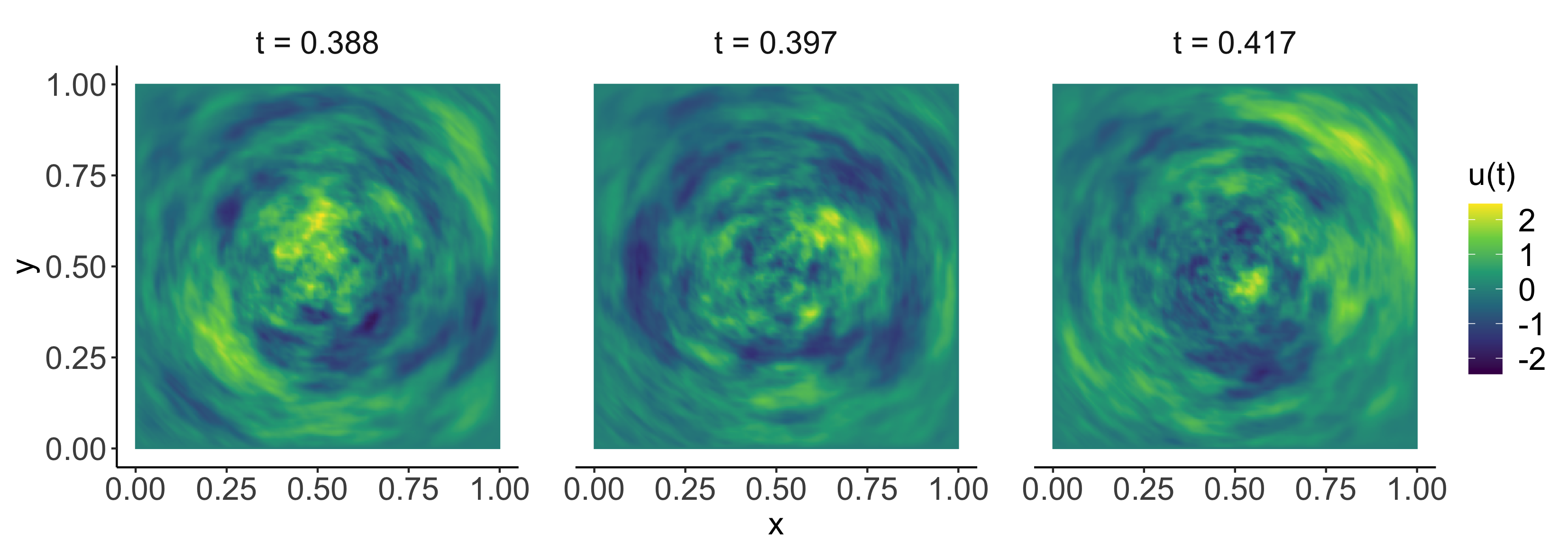}
    }\\
    \subcaptionbox{$\gamma = 0.75$}{
        \includegraphics[width = 16cm]{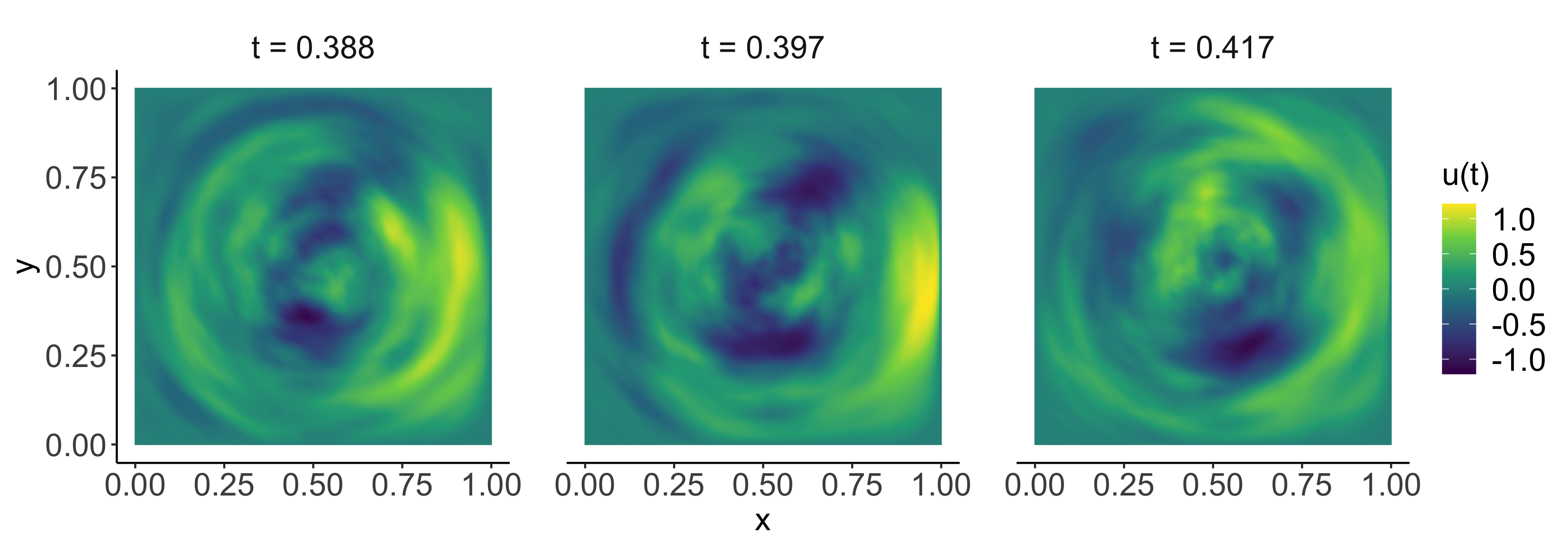}
    } \\
    \caption{Realizations of the SPDE \eqref{ga:eq:NonStatSPDE} as described in Example \ref{ga:example3}.\label{ga:example3:fig:NonStat}} 
\end{figure}

The main results of this paper can be summarized as follows:
\begin{enumerate}
    \item we propose a new and computationally efficient way of discretizing \eqref{eq:model-in-intro}, which contrary to existing discretizations, involves a finite element approximation of the additive noise,
    \item we derive strong and pathwise convergence rates for our proposed discretization, which are given in Theorem \ref{theorem:strong-rate} and Corollary \ref{cor:pathwise-rate}, respectively, and finally
    \item we verify the convergence rates obtained by numerical experiments using different values of $\gamma$ in one and two spatial dimensions.
\end{enumerate}
The paper is structured as follows: in Section 2 we state our assumptions on \eqref{eq:model-in-intro} (Assumption \ref{assumption:model}) and its numerical approximation (Assumption \ref{assumption:fem}). We also state a space and time regularity result of the mild solution of \eqref{eq:model-in-intro} (Lemma \ref{lemma:space-regularity-mild-solution} and Lemma \ref{lemma:time-regularity-mild-solution}). In Section 3 we outline our finite element approximation of \eqref{eq:model-in-intro}. We also state our strong and pathwise convergence results for this approximation (Theorem \ref{theorem:strong-rate} and Corollary \ref{cor:pathwise-rate}). Section 4 contains the proof of Theorem \ref{theorem:strong-rate}, and Section 5 numerical experiments verifying the convergence rate obtained in Theorem \ref{theorem:strong-rate} and Corollary \ref{cor:pathwise-rate}.

\section{Preliminaries and notation}

For separable Hilbert spaces $U, H$, we denote by $L(U,H)$ the Banach space of bounded linear operators from $U$ to $H$ with the usual norm, with convention $L(H) := L(H,H)$. We denote by $L_2(U,H)$ the Hilbert space of Hilbert--Schmidt operators from $U$ to $H$, with inner product,
\begin{align*}
    (A,B)_{L_2(H)} := \sum_{j = 1}^{\infty}(A e_j, B e_j)_H,
\end{align*}
for any orthonormal basis $\{ e_j \}_{j = 1}^{\infty}$ of $U$, with convention $L_2(H) := L_2(H,H)$. Throughout the paper, we will denote by $C$ a generic constant, which may change from line to line. Parameter dependence of $C$ will be denoted by subscripts, but is omitted unless relevant. We will also frequently and without further explanation consider $B \in L(H)$ with the property that $BA \in L(H)$ for some possibly unbounded and densely defined linear $A : D(A) \subseteq H \to H$, with a bounded inverse. By this we understand that $B A$ extends to $H$ from $D(A)$ as a bounded linear operator, and this extension is denoted by $BA$.

In what follows, we fix a filtered probability space, $(\Omega, \mathcal{F},\{\mathcal{F}_t\}_{t \in [0,T]}, P)$, fulfilling the usual conditions, and we let $W$ be a cylindrical Wiener process on $H$ (covariance operator $I$) adapted to the filtration $\{\mathcal{F}_t\}_{t \in [0,T]}$. Whenever we consider Itô integrals in the following, it will involve this cylindrical Wiener process. Further, we denote by $L^p(\Omega ; H)$ $(p \geq 1)$ the Banach space of equivalence classes of measurable functions $(\Omega, \mathcal{F}) \to (H, \mathcal{B}(H))$, with norm
\begin{align*}
    \Vert f \Vert_{L^p(\Omega ; H)}^p := E[\Vert f \Vert_H^p].
\end{align*}

Finally, we recall some properties of generators of variational semigroups. For another Hilbert space $V$ with $V \subseteq H$ densely and continuously, let $a : V \times V \to \mathbb{R}$ be a continuous bilinear form, with the property that the shifted bilinear form, $a(\cdot, \cdot) + \lambda (\cdot, \cdot)_H$, is coercive for $\lambda \geq 0$ large enough. The operator, $-A : D(A) \to H$, $D(A) \subseteq V$ defined by $(A u, v)_H = a(u, v)$ for any $u \in D(A)$ and $v \in V$ is called the generator of a variational semigroup on $H$, and is a subclass of generators of analytic semigroups. This semigroup is defined by the Dunford integral,
\begin{align*}
    S(t) := \frac{1}{2\pi i} \int_{-\lambda + \gamma} e^{-zt} (z - A)^{-1} \, dz,
\end{align*}
where $\gamma := \{ s e^{\pm i \delta}, \ s \geq 0 \}$ for some $\delta \in (0, \pi / 2)$ large enough. We can define fractional powers of $\lambda + A$ by, 
\begin{align}\label{eq:fractional-powers}
\begin{split}
    (\lambda + A)^{-\alpha} &:= \frac{1}{\Gamma(\alpha)} \int_0^{\infty} t^{-1 + \alpha} e^{-\lambda t} S(t) \, dt, \\
    (\lambda + A)^{\alpha} &:= \frac{\sin(\alpha \pi)}{\pi} \int_0^{\infty} t^{-1 + \alpha} (t + \lambda + A)^{-1} (\lambda + A) \, dt, 
\end{split}
\end{align}
for $\alpha \in (0,1)$. See e.g. Equation 6.9 and Theorem 6.9 in \cite{pazy} (the expression for $(\lambda + A)^{-\alpha}$ also holds for any $\alpha \geq 1$). The following properties of generators of analytic semigroups will be used frequently. 
\begin{lemma}\label{lemma:analytic-semigroup}
    Let $\lambda$, $A$ be as above, and $S(\cdot)$ the analytic semigroup generated by $-A$. Then,
    \begin{enumerate}
        \item[(a)] there is $C, \delta > 0$ such that for any $\alpha \geq 0$, $t > 0$ and $u \in H$
        \begin{align*}
            S(t) u \in D((\lambda + A)^{\alpha}) \quad \text{with} \quad \Vert (\lambda + A)^{\alpha} S(t) u \Vert_H \leq C e^{(\lambda-\delta)t} t^{-\alpha} \Vert u \Vert_H,
        \end{align*}
        with convention $(\lambda + A)^0 = I$, \medskip

        \item[(b)] for $\alpha, \beta \in \mathbb{R}$, $u \in D((\lambda + A)^{\gamma})$ with $\gamma = \max(\alpha, \beta, \alpha + \beta)$,
        \begin{align*}
            (\lambda + A)^{\alpha} (\lambda + A)^{\beta} u = (\lambda + A)^{\beta} (\lambda + A)^{\alpha} u = (\lambda + A)^{\alpha + \beta} u,
        \end{align*}
        
        \item[(c)] for $\alpha \in \mathbb{R}$, $D((\lambda + A)^{\alpha})$ is dense in $H$, \medskip
        
        \item[(d)] for $\alpha \in \mathbb{R}$, $(\lambda + A)^{\alpha} S(t) u = S(t) (\lambda + A)^{\alpha} u$ for $u \in D((\lambda + A)^{\alpha})$, \medskip
        
        \item[(e)] and finally for $\alpha \in [0,1]$,
        \begin{align*}
            \Vert (S(t) - I) u \Vert_H &\leq C e^{\lambda t} t^{\alpha} \Vert (\lambda + A)^{\alpha} u \Vert_H.
        \end{align*}
    \end{enumerate}
\end{lemma}
\begin{proof}
    The first four statements can be found in e.g. Chapter 2 (Theorem 6.8 and 6.13) in \cite{pazy}. The last one follows by similar arguments as in the case of $\lambda = 0$ (see e.g. \cite{2024-auestad}). 
\end{proof}

In what follows, we let $\mathcal{D} \subseteq \mathbb{R}^d$, $d \leq 3$ be a bounded polygonal domain, $H := L^2(\mathcal{D})$ and $V$ equal to some subspace of $H^1(\mathcal{D})$, where $H^s(\mathcal{D})$, $s \in \mathbb{N}$, is defined as the completion of $C^{\infty}(\mathcal{D})$ using the norm given by the inner product, $(u, v)_{H^s} := \sum_{\vert \alpha \vert \leq s} (\partial^{\alpha} u, \partial^{\alpha} v)_{L^2}$ where $\alpha \in \mathbb{N}_0^d$ is a multi-index.

\subsection*{Model and mild solution}

We make the following assumptions on \eqref{eq:model-in-intro}.
\begin{assumption}\label{assumption:model}
\hfill
\begin{description}
    \myitem{(M1)} $A_j : D(A_j) \to H$ in \eqref{eq:model-in-intro} are related to the bilinear forms $a_j : V \times V \to \mathbb{R}$, where, 
    \begin{align*}
        a_j(u, v) := \int_{\mathcal D} \mathcal{A}_j \nabla u \cdot \nabla v + (b_j \cdot \nabla u) v + \alpha_j uv \, d x, \quad j = 1, 2,
    \end{align*} 
    with $\mathcal{A}_j(x) \in \mathbb{R}^{d\times d}$, $b_j(x) \in \mathbb{R}^d$ and $\alpha_j(x) \in \mathbb{R}$. \medskip 
    \myitem{(M2)}\label{cond:continuity-coercivity} There is $\lambda \geq 0$ such that the shifted bilinear form, $a_1(\cdot, \cdot) + \lambda (\cdot,\cdot)_H$, is coercive and continuous on $V$, while $a_2(\cdot,\cdot)$ is coercive and continuous on $V$. \medskip
    \myitem{(M3)}\label{cond:hilbert-schmidt} For any $\alpha < -d/4$, there is $C_{\alpha} > 0$, such that $\Vert A_2^{\alpha} \Vert_{L_2(H)} \leq C_{\alpha}$. \medskip
    \myitem{(M4)}\label{cond:A1A2} For some $C > 0$, $\Vert (\lambda + A_1) A_2^{-1} \Vert_{L(H)} \leq C$ and $\Vert (\lambda + A_1)^{-1/2} A_2^{1/2} \Vert_{L(H)} \leq C$.
    \medskip
    \myitem{(M5)}\label{cond:gamma-size} $\gamma > d/4 - 1/2$. \medskip
    \myitem{(M6)} $\xi$ is $\mathcal{F}_0$-measurable and in $L^p(\Omega ; H)$ for some $p \geq 2$. \medskip
\end{description}
\end{assumption}

\begin{remark} \hfill
\begin{enumerate}[(a)]
    \item
    A sufficient condition for $a_j(\cdot, \cdot), \ j = 1, 2$ to satisfy \ref{cond:continuity-coercivity} is,
    \begin{enumerate}[(i)]
    \item $\vert \mathcal{A}_j \vert$, $\vert b_j \vert$ and $\alpha_j$ are in $L^{\infty}(\mathcal D)$, $j = 1, 2$, where $\vert \cdot \vert$ is the Euclidean norm, and 
    \item there is $c > 0$ such that for any $y \in \mathbb{R}^d$, and $y_0 \in \mathbb{R}$, 
    \begin{align}
    \label{eq:elliptic-condition}
        &y^T \mathcal{A}_1(x) y \geq c \vert y \vert^2, \qquad &&\text{for a.e. } x \in \mathcal{D}, \\
    \label{eq:coercive-condition}
        &y^T \mathcal{A}_2(x) y + b_2(x) \cdot y y_0 + \alpha_2(x) y_0^2 \geq c (\vert y \vert^2 + y_0^2), \qquad &&\text{for a.e. } x \in \mathcal{D}.
    \end{align}
    \end{enumerate}
    in which case one sees that 
    \begin{align*}
        a_j(u, v) \leq C \Vert u \Vert_{H^1} \Vert v \Vert_{H^1}, \quad a_1(u, u) + \lambda \Vert u \Vert_{L^2} \geq \frac{c}{2} \Vert u \Vert_{H^1}, \quad a_2(u, u) \geq c \Vert u \Vert_{H^1}^2,
    \end{align*}
    with $C = \Vert \mathcal{A}_j \Vert_{L^{\infty}} + \Vert b_j \Vert_{L^{\infty}} + \Vert \alpha_j \Vert_{L^{\infty}}$ (where $\Vert \mathcal{A}_j \Vert_{L^{\infty}} := \mathrm{ess \ sup}_{x \in \mathcal{D}} \vert \mathcal{A}_j(x) \vert$, and similarly for the other norms), provided $\lambda \geq \Vert \alpha_1 \Vert_{L^{\infty}} + \frac{1}{2c} \Vert b_1 \Vert_{L^{\infty}} + \frac{c}{2}$. If $V$ is $H_0^1(\mathcal D)$, the subspace of $H^1(\mathcal D)$ with zero Dirichlet boundary conditions, \eqref{eq:coercive-condition} can be replaced by the weaker condition $y^T \mathcal{A}_2(x) y + b_2(x) \cdot y y_0 + \alpha_2(x) y_0^2 \geq c \vert y \vert^2$ for a.e. $x \in \mathcal{D}$.

    \item
    Condition \ref{cond:hilbert-schmidt} holds in the case of $\mathcal{D} = (0,1)^d$, when $A_2 = -\Delta$ with zero Dirichlet boundary conditions, or $A_2 = I - \Delta$ with zero Neumann boundary conditions. Ordering and enumerating the eigenvalues $\lambda_j$ of $A_2$ in these cases, we find
    \begin{align*}
        c j^{2/d} \leq \lambda_j \leq C j^{2/d}, \quad j = 1, \dots, 
    \end{align*}
    for some $C, c > 0$, which in turn gives for any $\epsilon > 0$ small, 
    \begin{align*}
        \Vert (-\Delta)^{-d/4 - \epsilon} \Vert_{L_2(H)}^2 \leq C \sum_{j = 1}^{\infty} j^{\frac{4}{d} (-\frac{d}{4} - \epsilon)} < \infty. 
    \end{align*}
    More generally, it may be seen that this condition holds for $A_2$ selfadjoint in the case of general domains $\mathcal D$ where the coefficients of $\mathcal{A}_2$ and $\alpha_2$ are only required to be bounded (see Theorem 8.37 in \cite{gilbarg} and Theorem 6.3.1 in \cite{davies}). In the case of smooth domains $\mathcal D$ and when $\mathcal A_2$ has entries in $C^1(\overline{\mathcal D})$ and $\alpha_2$ is in $C(\overline{\mathcal D})$, see also Theorem 7.13 and Remark C.16 in \cite{haroske}. When $b_2 \neq 0$ and $A_2$ is not selfadjoint, some results are given in Theorem 7.15 in \cite{haroske}. 
    
    \item
    The first inequality of \ref{cond:A1A2} holds if $D(\lambda + A_1) = D(A_2) = H^2(\mathcal{D}) \cap V$, and $A_2$ satisfies the elliptic regularity estimate $\Vert A_2^{-1} f \Vert_{H^2} \leq C \Vert f \Vert_{L^2}$ for any $f \in L^2$. This is e.g. the case when $\mathcal{D}$ is convex, and $A_1$ and $A_2$ are selfadjoint with $C^1(\overline{\mathcal{D}})$ coefficients (see Remark 2(a) for precise conditions). The second inequality holds if $A_1$ and $A_2$ are selfadjoint and have the same boundary conditions. Owing to \ref{cond:continuity-coercivity} we have norm equivalence between $\Vert (\lambda + A_1)^{1/2} \cdot \Vert_{L^2}, \Vert A_2^{1/2} \cdot \Vert_{L^2}$ on $V$ (as both are equivalent to the $H^1$-norm), and an argument similar to that in Remark 2(c) below gives the inequality in this case. 
\end{enumerate}
\end{remark}

We are interested in the mild solution of \eqref{eq:model-in-intro}, which is defined by the stochastic convolution
\begin{align}\label{eq:mild-form}
    u(t) = S_1(t) \xi + \int_0^{t} S_1(t - s) A_2^{-\gamma} dW, \quad P\text{-a.s. for any } t \in [0,T],
\end{align}
where $S_1(\cdot)$ is the semigroup generated by $-A_1$. The next lemma asserts that we have a mild solution under Assumption \ref{assumption:model}. The proof is standard, and can be found in Appendix \ref{app:prelim-proofs}.
\begin{lemma}\label{lemma:mild-solution-existence}
    Under Assumption \ref{assumption:model}, the model \eqref{eq:model-in-intro} has a unique mild solution in $L^p(\Omega ; H)$ for any $p \geq 2$. Further, for any $\rho \leq 0$ there is $C_p, \delta > 0$ such that
    \begin{align*}
        \Vert u(t) \Vert_{L^p(\Omega ; H)} \leq C_p e^{\lambda t} (1 + e^{-\delta t} t^{\rho}\Vert (\lambda + A_1)^{\rho} \xi \Vert_{L^p(\Omega ; H)}).
    \end{align*}
\end{lemma}
The following two lemmas describe the space and time regularity of the mild solution, respectively. The proofs are standard, and the details can be found in Appendix \ref{app:prelim-proofs}.
\begin{lemma}\label{lemma:space-regularity-mild-solution}
    Suppose Assumption \ref{assumption:model} holds, and let $u$ be the mild solution to \eqref{eq:model-in-intro}. For $p \geq 2$, $\alpha \in [0, \gamma + 1/2 - d/4) \cap [0,3/2)$, and $\rho \leq \alpha$, there is $C_{p,\alpha}, \delta > 0$ such that,
    \begin{align*}
        \Vert (\lambda + A_1)^{\alpha} u(t) \Vert_{L^p(\Omega ; H)} \leq C_{p,\alpha} e^{\lambda t} (1 + e^{-\delta t} t^{-\alpha + \rho} \Vert (\lambda + A_1)^{\rho} \xi \Vert_{L^p(\Omega ; H)}),
    \end{align*}
    As a consequence, $u(t) \in D((\lambda + A_1)^{\alpha}), \, t > 0$ $P$-a.s. for any $\alpha$ in this interval. 
\end{lemma}
\begin{lemma}\label{lemma:time-regularity-mild-solution}
    Suppose Assumption \ref{assumption:model} holds, and let $u$ be the mild solution to \eqref{eq:model-in-intro}. For $p \geq 2$, $0 \leq t_1 \leq t_2$, $\alpha \in [0,\gamma + 1/2 - d / 4) \cap [0,1/2]$ and $\rho \leq \alpha$, there is $C_{p,\alpha}, \delta > 0$ such that, 
    \begin{align*}
        \Vert u(t_2) - u(t_1) \Vert_{L^p(\Omega ; H)} \leq C_{p,\alpha} e^{\lambda t_2} (t_2-t_1)^{\alpha} (1 + e^{-\delta t_1} t_1^{-\alpha + \rho} \Vert (\lambda +  A_1)^{\rho} \xi \Vert_{L^p(\Omega ; H)}).
    \end{align*}
    As a consequence, $u$ has $\beta$-Hölder continuous trajectories $P$-a.s. for any $\beta \in [0,\gamma + 1/2 - d/4) \cap [0,1/2)$. 
\end{lemma}

\subsection*{Finite element approximation}
In order to approximate $A_j$, we introduce finite dimensional subspaces $V_h \subseteq V$, with $N_h := \mathrm{dim}(V_h)$, consisting of piecewise first order polynomials defined on a regular collection of simplices with maximum diameter $h$. We make the following standard assumption on $V_h$ and the finite element approximation $A_{j,h}$ of $A_j$.
\begin{assumption}\label{assumption:fem}
\hfill
\begin{description}
    \myitem{(N1)}\label{cond:quasi-uniform} $V_h \subseteq V$ consists of continuous functions that are first order polynomials when restricted to $\tau \in \mathcal{T}_h$, where $\mathcal{T}_h$ is a collection of simplices with disjoint interior, satisfying $\bigcup_{\tau \in \mathcal{T}_h} \tau = \overline{\mathcal{D}}$, and for some $C > 0$,
    \begin{align*}
        C^{-1} h \leq 2 \rho(\tau) \leq \text{diam}(\tau) \leq 2 r(\tau) \leq C h, \quad \tau \in \mathcal{T}_h,
    \end{align*}
    where $\rho, r$ are the radii of the incircle and circumcircle, respectively. \medskip
    \myitem{(N2)}\label{cond:discrete-operator} The finite dimensional operators $A_{j,h} : V_h \to V_h$ are defined by
    \begin{align*}
        (A_{j,h} u, v)_H = a_j(u,v), \quad \text{for any } u, v \in V_h.
    \end{align*}
    \myitem{(N3)}\label{cond:solution-operator} For some $\lambda \geq 0$, and with $\pi_h$ the $H$-orthogonal projection onto $V_h$,
    \begin{align*}
        \Vert ((\lambda + A_1)^{-1} - (\lambda + A_{1,h})^{-1} \pi_h) (\lambda + A_1)^{\alpha} f \Vert_H \leq C h^{2-2\alpha} \Vert f \Vert_H, 
    \end{align*}
    and, 
    \begin{align*}
        \Vert A_2^{\alpha} (A_2^{-1} - A_{2,h}^{-1} \pi_h) f \Vert_H \leq C h^{2 - 2\alpha} \Vert f \Vert_H,
    \end{align*}
    for $\alpha \in \{0, 1/2\}$, and some $C > 0$. \medskip
    \myitem{(N4)}\label{cond:fem-bound} There is $C > 0$ such that for $f \in H$, $\Vert (\lambda + A_{1,h})^{-1/2} \pi_h (\lambda + A_1)^{1/2} f \Vert_H \leq C \Vert f \Vert_H$.
\end{description}
\end{assumption}
\begin{remark}\label{remark:solution-operator} \hfill
\begin{enumerate}[(a)]
    \item
    We interpret \ref{cond:solution-operator} as a condition on the rate of convergence of the error $u - u_h$, where $(\lambda + A_1) u = f$ while $(\lambda + A_{1,h}) u_h = \pi_h f$. In the first inequality of \ref{cond:solution-operator}, different $\alpha$ corresponds to different regularity of $f$, while in the second it corresponds to different norms used for the difference $u - u_h$. It is also common, like in Chapter 7.3 in \cite{yagi}, to formulate condition \ref{cond:solution-operator} in terms of the Ritz projection. For $A_1$, this would be defined by $R_h := (\lambda + A_{1,h})^{-1} \pi_h (\lambda + A_1)$. Sufficient conditions for \ref{cond:solution-operator} to hold are for example:
    \begin{enumerate}[(i)]
        \item $\mathcal{D} \subseteq \mathbb{R}^d, d \leq 3$, is bounded, polygonal and convex,
        \item $V$ is the closed subspace of $H^1(\mathcal D)$ with either zero Neumann or Dirichlet boundary conditions,
        \item $A_j$ is selfadjoint,
        \item the entries of $\mathcal{A}_j$ are in $C^1(\overline{\mathcal{D}})$, and $\alpha_j$ are in $L^{\infty}(\mathcal D)$,
        \item there is $c, c_0 > 0$ such that $y^T \mathcal{A}_j(x) y \geq c \vert y \vert^2$ for any $x \in \overline{\mathcal D}$, $y \in \mathbb{R}^2$, while $\alpha_2 \geq c_0 > 0$ for a.e. $x \in \mathcal{D}$.
    \end{enumerate}
    For the case of $d = 2$ and zero Neumann boundary conditions see e.g. subsection 7.3.4 in \cite{yagi}, in particular Equation 7.58. The other cases follow by similar computations, noting that for $d \leq 3$ the interpolant on the finite element mesh, $I_h f$, of $f \in H^2(\mathcal{D}) \subseteq C(\overline{\mathcal{D}})$, satisfies $\Vert f - I_h f \Vert_{H^1} \leq C h \Vert f \Vert_{H^2}$ (see e.g. Theorem 4.4.4 in \cite{brenner}).
    \item
    Condition \ref{cond:fem-bound} holds in the case of $A_1$ selfadjoint. To see this, assume without loss of generality that $\lambda = 0$. Note that for any $u \in V_h$, the following identity holds
    \begin{align}\label{eq:remark-fem-bound}
        \Vert A_{1,h}^{1/2} u \Vert_H^2 = (A_{1,h}^{1/2} u, A_{1,h}^{1/2} u)_H = (A_{1,h} u, u)_H = a_1(u,u) = \Vert A_1^{1/2} u \Vert_H^2. 
    \end{align}
    Therefore, by the Cauchy--Schwarz inequality and \eqref{eq:remark-fem-bound}, we have for any $f \in H$ (using that for any $g \in V_h$ we have $\Vert g \Vert_H = \sup_{\Vert \varphi \Vert_H = 1, \varphi \in V_h} (g, \varphi)_H$\footnote{By choosing $\varphi = g / \Vert g \Vert_H \in V_h$ we find $\sup_{\Vert \varphi \Vert_H = 1, \varphi \in V_h} (g, \varphi)_H \geq \Vert g \Vert_H$, while $\sup_{\Vert \varphi \Vert_H = 1, \varphi \in V_h} (g, \varphi)_H \leq \sup_{\Vert \varphi \Vert_H = 1} (g, \varphi)_H = \Vert g \Vert_H$.}),
    \begin{align*}
        \Vert A_{1,h}^{-1/2} \pi_h A^{1/2} f \Vert_H &= \sup_{\Vert \varphi \Vert_H = 1, \varphi \in V_h} (A_{1,h}^{-1/2} \pi_h A_1^{1/2} f, \varphi)_H \\
        &= \sup_{\Vert \varphi \Vert_H = 1, \varphi \in V_h} (f, A_1^{1/2} A_{1,h}^{-1/2} \varphi)_H \\
        &\leq \sup_{\Vert \varphi \Vert_H = 1, \varphi \in V_h} \Vert f \Vert_H \Vert A_1^{1/2} A_{1,h}^{-1/2} \varphi \Vert_H \\
        &\leq \Vert f \Vert_H.
    \end{align*}
    \item
    A consequence of \ref{cond:quasi-uniform} is the estimate,
    \begin{align*}
        \Vert u \Vert_{H^1(\mathcal D)} \leq C h^{-1} \Vert u \Vert_H, \quad u \in V_h,
    \end{align*}
    for some $C > 0$ independent of $h$ (see e.g. Theorem 3.2.6 in \cite{ciarlet}). This in turn gives for any $u \in V_h$, using \ref{cond:discrete-operator} and the continuity of the bilinear forms $a_j$, 
    \begin{align*}
        \Vert A_{j,h} u \Vert_H &= \sup_{\Vert \varphi \Vert_H = 1, \varphi \in V_h} (A_{j,h} u, \varphi)_H \\
        &= \sup_{\Vert \varphi \Vert_H = 1, \varphi \in V_h} a_j(u, \varphi) \\ 
        &\leq \sup_{\Vert \varphi \Vert_H = 1, \varphi \in V_h} C \Vert u \Vert_{H^1(\mathcal D)} \Vert \varphi \Vert_{H^1(\mathcal D)} \\
        &\leq C h^{-2} \Vert u \Vert_H.
    \end{align*}
\end{enumerate}
\end{remark}

\subsection*{Quadrature approximation of fractional operator}

In order to approximate the fractional power operator, we use the quadrature in \cite{2019-bonito}. To that end, let
\begin{align}\label{eq:quadrature}
    Q_k^{-\gamma}(A_2) := \frac{k \sin(\pi \gamma)}{\pi} \sum_{j = -M}^N e^{(1-\gamma) y_j } (e^{y_j} I + A_2)^{-1}.
\end{align}
Here, $k > 0$ is the quadrature resolution, $y_j = j k, \ j = -M, \dots, N$, and
\begin{align*}
    N = \bigg\lceil \frac{\pi^2}{2 \gamma k^2} \bigg\rceil, \quad M = \bigg\lceil \frac{\pi^2 }{2(1 - \gamma) k^2} \bigg\rceil.
\end{align*}
The following lemma describe the convergence rate of this approximation. 
\begin{lemma}\label{lemma:A_h-quadrature}
Let $A_{2,h}$ be as in \ref{cond:discrete-operator}, $Q_k^{-\gamma}$ be given by \eqref{eq:quadrature}, and $\gamma \in (0,1)$. Then, for some $C, c > 0$ independent of $k$ and $h$, the following estimate holds,
\begin{align}
    \Vert (A_{2,h}^{-\gamma} - Q_k^{-\gamma}(A_{2,h})) u \Vert_H \leq C e^{-c/k} \Vert u \Vert_H, \quad \text{for any } u \in V_h.
\end{align}
\end{lemma}
\begin{proof}
    As stated in Theorem 3.2 in \cite{2019-bonito}, the constant $C$ depends only on $\gamma$ and the continuity and coercivity constant of $(A_{2,h} \cdot, \cdot)_H$ on $V_h$, which is the same as that of $a_2(\cdot,\cdot)$. The constant $c$ is given explicitly in Theorem 3.2 in \cite{2019-bonito}, and does not depend on $h$.  
\end{proof}

\section{Numerical method and convergence results}

In this section we present our proposed finite element approximation, and state our convergence results. Our semidiscrete approximation is based on the SPDE \eqref{eq:mild-solution-semi-approx} with values in $V_h$. When $\gamma \in (0,1)$ we approximate the fractional power operator by the quadrature approximation \eqref{eq:quadrature}, in which case \eqref{eq:mild-solution-semi-approx} becomes
\begin{align}\label{eq:semidiscrete-gamma-in-0-1}
    d u_h = -A_{1,h} u_h \, dt + Q_k^{-\gamma}(A_{2,h}) \pi_h \, dW,
\end{align}
with convention $Q_k^{-1}(A_{2,h}) = A_{2,h}^{-1}$ and $Q_k^0(A_{2,h}) = I$. Discretizing \eqref{eq:semidiscrete-gamma-in-0-1} in time with backward Euler, we get our fully discrete approximation,
\begin{align}\label{eq:fully-discrete-scheme}
    (I + \Delta t A_{1,h})u_{h,\Delta t}(t_{n+1}) = u_{h,\Delta t}(t_n) + Q_k^{-\gamma}(A_{2,h}) \pi_h (W(t_{n+1}) - W(t_n)),
\end{align}
where $t_n = n \Delta t$.

In order to simulate from \eqref{eq:fully-discrete-scheme} we rewrite this equation as a system of equations for the coefficients of $u_{h,\Delta t}$ in the nodal basis of $V_h$. To that end, we denote this basis by $\varphi_j, \ j = 1, \dots, N_h$, and let
\begin{align}\label{eq:fem-matrices}
    (M_h)_{ij} = \int_{\mathcal D} \varphi_j \varphi_i \, dx, \quad (T_{h})_{ij} = a_{1,h}(\varphi_j, \varphi_i), \quad (K_h)_{ij} = a_{2,h}(\varphi_j, \varphi_i).
\end{align}
We may express $u_{h,\Delta t}(t_n) = \sum_{j = 1}^{N_h} \alpha_j^n \varphi_j$, for coefficients $\alpha_j^n$, and using Lemma \ref{lemma:discrete-cylindrical-wiener-process} (see Appendix \ref{app:derivation-of-system} for a detailed derivation), we note that \eqref{eq:fully-discrete-scheme} may be rewritten
\begin{align}\label{eq:scheme-basis-coefficients-1}
    (M_h + \Delta t T_h) \alpha^{n+1} = M_h \alpha^n + M_h \Delta t^{1/2} \frac{k \sin(\pi \gamma)}{\pi} \sum_{j = -M}^N e^{(1-\gamma) y_j} (e^{y_j} M_h + K_h)^{-1} M_h^{1/2} \varrho^n,
\end{align}
for $\gamma \in (0,1)$, and
\begin{align}\label{eq:scheme-basis-coefficients-2}
    (M_h + \Delta t T_{h}) \alpha^{n+1} = M_h \alpha^{n} + M_h \Delta t^{1/2} K_{h}^{-1} M_h^{1/2} \varrho^n,
\end{align}
for $\gamma = 1$. Here, $\varrho^n \sim \mathcal{N}(0,I)$ are $N_h$-dimensional Gaussian and independent for each $n$. 

The following theorem is the main result of this paper, and describes rates of strong convergence of our approximation \eqref{eq:fully-discrete-scheme}.
\begin{theorem}\label{theorem:strong-rate}
Suppose,
\begin{enumerate}
    \item Assumption \ref{assumption:model} and \ref{assumption:fem} hold, \medskip
    \item $\gamma \in (d/4 - 1/2, 1] \cap [0,1]$, and \medskip
    \item the quadrature resolution $k \leq -c(2 \gamma + 1)^{-1} \log(h)^{-1}$, where $c$ is as in Lemma \ref{lemma:A_h-quadrature},
\end{enumerate}
and let $u_{h,\Delta t}$ and $u$ be the solutions to \eqref{eq:fully-discrete-scheme} and \eqref{eq:mild-form}, respectively. Then for any $p \geq 2$, $\theta \in [0, 2 \gamma + 1 - d/2) \cap [0,2)$ and $\rho \leq \theta$, there is $C_{p,\theta}, c > 0$ such that
\begin{align*}
    \Vert u(t) - u_{h, \Delta t}(t) \Vert_{L^p(\Omega; H)} \leq C_{p,\theta} e^{c \lambda t} (h^{\theta} + \Delta t^{\theta / 2}) (1 + t^{-\theta / 2 + \rho / 2} \Vert (\lambda + A_1)^{\rho / 2} \xi \Vert_{L^p(\Omega ; H)}).
\end{align*}
\end{theorem}
The following corollary describes the pathwise convergence. 
\begin{corollary}\label{cor:pathwise-rate}
    Suppose the conditions of Theorem \ref{theorem:strong-rate} hold. Then, for any $\theta \in [0, 2\gamma + 1 - d/2) \cap [0,2]$, $\epsilon > 0$ small, and sequences $h_n, \Delta t_n$ such that
\begin{align*}
    \sum_{n = 1}^{\infty} (h_n^{\theta} + \Delta t_n^{\theta/2})^{p\epsilon} < \infty,
\end{align*} 
there is a random variable $M_{\theta, \epsilon, t} > 0$ such that
\begin{align*}
    \Vert u(t) - u_{h_n, \Delta t_n}(t) \Vert_H \leq M_{\theta, \epsilon, t} (h_n^{\theta} + \Delta t_n^{\theta / 2})^{1-\epsilon}, \quad P\text{-a.s.}
\end{align*}
\end{corollary}
\begin{proof}
    This follows by Theorem \ref{theorem:strong-rate} combined with Lemma \ref{lemma:pathwise} below. 
\end{proof}

\section{Proof of Theorem \ref{theorem:strong-rate}}

To prove Theorem \ref{theorem:strong-rate}, we need a couple of lemmas. Let $S_{h,\Delta t}(\cdot)$ be our fully discrete approximation of $S_1(\cdot)$ based on backward Euler, defined by
\begin{align}\label{eq:fully-discrete-semigroup-operator}
    S_{h,\Delta t}(t) :=
    \begin{cases}
        I, \quad &t = 0, \\
        \sum_{n = 0}^{N-1} r(\Delta t A_{1,h} )^{n+1} \chi_{(t_n,t_{n+1}]}(t), \quad &t > 0,
    \end{cases}
\end{align}
where $r(z) := (1 + z)^{-1}$, $\chi_D(\cdot)$ is the indicator function on $D$, and $N > 0$. The following lemma describes the error of this approximation. 
\begin{lemma}\label{lemma:fully-discrete-semigroup-approximation}
    Let $S_{h,\Delta t}(t)$ be given as in \eqref{eq:fully-discrete-semigroup-operator}. Under Assumption \ref{assumption:model} and \ref{assumption:fem}, there is $C, c, \delta > 0$ such that
    \begin{align*}
        \Vert (S_1(t) - S_{h,\Delta t}(t) \pi_h ) u_0 \Vert_H &\leq C e^{c (\lambda-\delta) t} t^{-\theta/2 + \rho/2} (h^{\theta} + \Delta t^{\theta / 2}) \\
        &\qquad \times \min(\Vert (\lambda + A_1)^{\rho/2} u_0 \Vert_H, \Vert A_2^{\rho / 2} u_0 \Vert_H),
    \end{align*}
    for $\theta \in [0,2]$, and $\rho \in [-1, \theta] \cap [-2 + \theta, \theta]$.
\end{lemma}
\begin{proof}
    This follows by Theorem 2.14 and 2.24 in \cite{2024-auestad}, noting that Assumption \ref{assumption:model} and \ref{assumption:fem} imply Assumption 2.1 and condition (A6)--(A9) in \cite{2024-auestad}. 
    
    We may replace $\lambda + A_1$ by $A_2$ since we can recover the interpolated inequality above by interpolating (see e.g. Section 2.8.3 and the Heinz--Kato inequality in \cite{yagi}) the operator $(S_1(t) - S_{h,\Delta t}(t) \pi_h ) (\lambda + A_1)^{\alpha}$ with $\alpha \in \{-1, 0, 1/2\}$, in which case, $(\lambda + A_1)^{\alpha}$ may be replaced by the corresponding powers of $A_2$ due to condition \ref{cond:A1A2}. 
\end{proof}
The next lemma describes a smoothing property for the fully discrete semigroup, $S_{h,\Delta t}(\cdot)$, and follows from Lemma \ref{lemma:analytic-semigroup}, \ref{lemma:fully-discrete-semigroup-approximation}, \ref{cond:fem-bound} and \ref{cond:A1A2}.
\begin{lemma}\label{lemma:fully-discrete-semigroup-smoothing}
    Under the conditions of Lemma \ref{lemma:fully-discrete-semigroup-approximation}, there is $C, \delta > 0$ such that
    \begin{align*}
        \Vert S_{h,\Delta t}(t) \pi_h A_2^{\alpha} \Vert_{L(H)} \leq C e^{c(\lambda - \delta) t} t^{-\alpha},
    \end{align*}
    for $\alpha \in [0, 1/2]$.
\end{lemma}
\begin{proof}
    By the proofs of Lemma 2.7 and 2.10 in \cite{2024-auestad}, condition \ref{cond:fem-bound} and \ref{cond:A1A2},
    \begin{align*}
        \Vert S_{h,\Delta t}(t) \pi_h A_2^{\alpha} \Vert_{L(H)} &\leq \Vert S_{h,\Delta t}(t) (\lambda + A_{1,h})^{\alpha} \pi_h \Vert_{L(H)} \Vert (\lambda + A_{1,h})^{-\alpha} \pi_h (\lambda + A_1)^{\alpha} \Vert_{L(H)} \\
        &\qquad \times \Vert (\lambda + A_1)^{-\alpha} A_2^{\alpha} \Vert_{L(H)} \\
        &\leq C e^{c (\lambda-\delta) t} t^{-\alpha},
    \end{align*}
    for $\alpha \in \{0,1/2\}$. Therefore, by interpolation, it holds for $\alpha \in [0,1/2]$.
\end{proof}
The following lemma is an error estimate for the semidiscrete approximation, $S_{2,h}(\cdot)$, of $S_2(\cdot)$. 
\begin{lemma}\label{lemma:new-norm-semidiscrete-error}
    Let $S_{2,h}(\cdot)$ be the analytic semigroup generated by $-A_{2,h}$ on $V_h$. Under Assumption \ref{assumption:model} and \ref{assumption:fem} there are $C, \delta > 0$ such that 
    \begin{align*}
        \Vert A_2^{\alpha} (S_2(t) - S_{2,h}(t) \pi_h) u_0 \Vert_{H} \leq C e^{-\delta t} t^{-\theta / 2 } h^{\theta - 2 \alpha} \Vert u_0 \Vert_H
    \end{align*}
    for any $\alpha \in [-1/2, 1/2]$, $\theta \in [2\alpha,2] \cap [0,2+2\alpha]$. 
\end{lemma}
\begin{proof}
    By similar arguments as in the proof of Lemma \ref{lemma:fully-discrete-semigroup-approximation}, this follows by Lemma 2.15 and 2.18 in \cite{2024-auestad}.
\end{proof}
The following lemma describes the error in our approximation of the stochastic convolution \eqref{eq:mild-form} due to the finite element approximation of the covariance operator.
\begin{lemma}\label{lemma:error-stochastic-convolution}
    Suppose the conditions of Theorem \ref{theorem:strong-rate} hold. Then for any $p \geq 2$ and $\theta \in [0, 2\gamma + 1 - d / 2) \cap [0,2]$, there is $C_{p,\theta}, c > 0$ such that
    \begin{align*}
        \Vert \int_0^t S_1(t-s) A_2^{-\gamma} \, dW - \int_0^t S_{h,\Delta t}(t-s) A_{2,h}^{-\gamma} \pi_h \, dW \Vert_{L^p(\Omega ; H)} \leq C_{p,\theta} e^{c \lambda t} (h^{\theta} + \Delta t^{\theta / 2}).
    \end{align*}
\end{lemma}
\begin{proof}
    We decompose the difference as follows
    \begin{align*}
        &\int_0^t S_1(t-s) A_2^{-\gamma} \, dW - \int_0^t S_{h,\Delta t}(t-s) A_{2,h}^{-\gamma} \pi_h \, dW \\
        &\qquad = \int_0^t (S_1(t-s) - S_{h,\Delta t}(t-s) \pi_h) A_2^{-\gamma} \, dW \quad (=: (i)) \\
        &\qquad \qquad + \int_0^t S_{h,\Delta t}(t-s) \pi_h (A_2^{-\gamma} - A_{2,h}^{-\gamma} \pi_h) \, dW. \quad (=: (ii))
    \end{align*}
    Using a version of the Burkholder--Davis--Gundy (BDG) inequality (see e.g. Theorem~4.36 in~\cite{da-prato}), we have for some $0 < \epsilon < (\gamma - d/4 + 1/2) / 2$
    \begin{align*}
        \Vert (i) \Vert_{L^p(\Omega ; H)}^2 &\leq C_p \int_0^t \Vert (S_1(t-s) - S_{h,\Delta t}(t-s) \pi_h) A_2^{-\gamma + d / 4 + \epsilon} A_2^{-d / 4 - \epsilon} \Vert_{L_2(H)}^2 \, ds \\
        &\leq C_p \int_0^t \Vert (S_1(t-s) - S_{h,\Delta t}(t-s) \pi_h) A_2^{-\gamma + d / 4 + \epsilon} \Vert_{L(H)}^2 \Vert A_2^{-d/4 - \epsilon} \Vert_{L_2(H)}^2 \, ds \\
        &\leq C_{p,\epsilon} e^{2 c \lambda t} \int_0^t ( e^{-\delta s} s^{-\theta / 2 + \gamma - d/4 - \epsilon} (h^{\theta} + \Delta t^{\theta / 2}))^2 \, ds \\
        &\leq C_{p,\epsilon} e^{2 c \lambda t} (h^{2 \gamma + 1 - d/2 - 4\epsilon} + \Delta t^{\gamma + 1 / 2 - d / 4 - 2\epsilon})^2,
    \end{align*}
    by Lemma \ref{lemma:fully-discrete-semigroup-approximation} with $\theta = 2\gamma + 1 - d/2 - 4\epsilon$ and \ref{cond:hilbert-schmidt}.
    
    For the second term, Lemma \ref{lemma:fully-discrete-semigroup-smoothing}, the BDG inequality (Theorem 4.36 in \cite{da-prato}), and the properties $\Vert L \Vert_{L_2(H)} = \Vert L^* \Vert_{L_2(H)}$, $\Vert L \Vert_{L(H)} = \Vert L^* \Vert_{L(H)}$ for any $L \in L_2(H)$ with adjoint $L^*$, gives for $0 < \epsilon < (\gamma - d / 4 + 1/2) / 2$
    \begin{align*}
        \Vert (ii) \Vert_{L^p(\Omega ; H)}^2 &\leq C_p \int_0^t \Vert S_{h,\Delta t}(t-s) \pi_h (A_2^{-\gamma} - A_{2,h}^{-\gamma} \pi_h) \Vert_{L_2(H)}^2 \, ds \\
        &= C_p \int_0^t \Vert S_{h,\Delta t}(t-s) \pi_h A_2^{1/2 - \epsilon} A_2^{-1/2 + \epsilon} (A_2^{-\gamma} - A_{2,h}^{-\gamma} \pi_h) \Vert_{L_2(H)}^2 \, ds \\
        &\leq C_p \int_0^t \Vert S_{h,\Delta t}(t-s) \pi_h A_2^{1/2 - \epsilon} \Vert_{L(H)}^2 \Vert A_2^{-1/2 + \epsilon} (A_2^{-\gamma} - A_{2,h}^{-\gamma} \pi_h) \Vert_{L_2(H)}^2 \, ds \\
        &\leq C_{p,\epsilon} e^{2 c \lambda t} \Vert A_2^{-1/2 + \epsilon} (A_2^{-\gamma} - A_{2,h}^{-\gamma} \pi_h) \Vert_{L_2(H)}^2 \\
        &= C_{p,\epsilon} e^{2 c \lambda t} \Vert (A_2^{d/4 -1/2 + 2\epsilon} (A_2^{-\gamma} - A_{2,h}^{-\gamma}) )^* (A_2^{-d/4-\epsilon})^* \Vert_{L_2(H)}^2 \\
        &\leq C_{p,\epsilon} e^{2 c \lambda t} \Vert A_2^{d/4 -1/2 + 2\epsilon} (A_2^{-\gamma} - A_{2,h}^{-\gamma}) \Vert_{L(H)}^2 \Vert A_2^{-d/4 - \epsilon} \Vert_{L_2(H)}^2 \\
        &\leq C_{p,\epsilon} e^{2 c \lambda t} \Vert A_2^{d/4 -1/2 + 2\epsilon} (A_2^{-\gamma} - A_{2,h}^{-\gamma}) \Vert_{L(H)}^2,
    \end{align*}
    due to condition \ref{cond:hilbert-schmidt}. By the definition of the negative fractional power \eqref{eq:fractional-powers} and Lemma \ref{lemma:new-norm-semidiscrete-error},
    \begin{align*}
        \Vert A_2^{d/4 -1/2 + 2\epsilon} (A_2^{-\gamma} - A_{2,h}^{-\gamma}) \Vert_{L(H)} &= \Vert A_2^{d/4 - 1/2 + 2 \epsilon} \int_0^{\infty} t^{-1 + \gamma} (S_2(t) - S_{2,h}(t) \pi_h) \, dt \Vert_{L(H)} \\
        &\leq \int_0^{\infty} t^{-1 + \gamma} \Vert A_2^{d/4 - 1/2 + 2 \epsilon} (S_2(t) - S_{2,h}(t) \pi_h) \Vert_{L(H)} \, dt \\
        &\leq C \int_0^{\infty} e^{-\delta t} t^{-1 + \gamma - \theta / 2} h^{\theta - 2(d/4-1/2+2\epsilon)} \, dt \\
        &\leq C_{\epsilon} h^{2 \gamma + 1 - d / 2 - 5 \epsilon},
    \end{align*}
    with $\theta = 2 \gamma - \epsilon$ in the last line, where we used that $A_2^{\alpha}$ is closed to pass it under the integral sign. This finishes the second term. 
\end{proof}

We are now ready to prove Theorem \ref{theorem:strong-rate}.
\begin{proof}[Proof of Theorem \ref{theorem:strong-rate}]
Our approximate mild solution \eqref{eq:fully-discrete-scheme} can be extended from discrete times to all times $t \geq 0$ as
\begin{align*}
    u_{h,\Delta t}(t) := S_{h,\Delta t}(t) \pi_h \xi + \int_0^t S_{h,\Delta t}(t-s) Q_k^{\gamma}(A_{2,h}) \pi_h \, dW.
\end{align*}
We decompose the error as follows
\begin{align*}
    u(t) - u_{h,\Delta t}(t) &= (S_1(t) - S_{h,\Delta t}(t) \pi_h) \xi \\
    &\qquad+ \bigg(\int_0^t S_1(t-s) A_2^{-\gamma} \, dW - \int_0^t S_{h,\Delta t}(t-s) A_{2,h}^{-\gamma} \pi_h \, dW \bigg) \\
    &\qquad+ \bigg(\int_0^t S_{h,\Delta t}(t-s) A_{2,h}^{-\gamma} \pi_h \, dW - \int_0^t S_{h,\Delta t}(t-s) Q_k^{-\gamma}(A_{2,h}) \pi_h \, dW \bigg) \\
    &=: (i) + (ii) + (iii).
\end{align*}
We have,
\begin{align*}
    \Vert (i) \Vert_{L^p(\Omega ; H)} \leq C e^{c \lambda t} t^{-\theta/2 + \rho/2} (h^{\theta} + \Delta t^{\theta/2}) \Vert (\lambda + A_1)^{\rho / 2} \xi \Vert_{L^p(\Omega ; H)},
\end{align*}
for any $\theta \in [0,2]$ and $\rho \leq \theta$, by Lemma \ref{lemma:fully-discrete-semigroup-approximation}. For $(ii)$, we use Lemma \ref{lemma:error-stochastic-convolution}. 

Finally for $(iii)$, we have using the BDG inequality (Theorem 4.36 in \cite{da-prato}), Lemma \ref{lemma:fully-discrete-semigroup-smoothing} and \ref{lemma:A_h-quadrature},
\begin{align*}
    &\Vert \int_0^t S_{h,\Delta t}(t-s) (A_{2,h}^{-\gamma} - Q_k^{-\gamma}(A_{2,h})) \pi_h \, dW \Vert_{L^p(\Omega; H)}^p \\
    &\qquad\leq C_p \bigg( \int_0^t \Vert S_{h,\Delta t}(t-s) (A_{2,h}^{-\gamma} - Q_k^{-\gamma}(A_{2,h})) \pi_h \Vert_{L_2(H)}^2 \, ds \bigg)^{p/2} \\
    &\qquad\leq C_p \bigg( \int_0^t e^{2c(\lambda - \delta) s} \Vert A_{2,h}^{-\gamma} - Q_k^{-\gamma}(A_{2,h}) \pi_h \Vert_{L_2(H)}^2 \, ds \bigg)^{p/2} \\
    &\qquad\leq C_p e^{p c \lambda t} \bigg( \int_0^t e^{-2c \delta s} \sum_{j=1}^{N_h} \Vert (A_{2,h}^{-\gamma} - Q_k^{-\gamma}(A_{2,h})) e_{h,j} \Vert_H^2 \, ds \bigg)^{p/2} \\
    &\qquad\leq C_p e^{p c \lambda t} (N_h^{1/2} e^{-c/k})^p, 
\end{align*}
where $\{ e_{h,j} \}_j$ is an $H$-orthonormal basis of $V_h$. Using the bound on $k$ in terms of $h$, the bound $N_h \leq C h^{-d}$ which follows from \ref{cond:quasi-uniform}, this term is bounded by $C (h^{2\gamma + 1 - d/2})^p$. If $\gamma = 1$, this term vanishes.

In total, we get,
\begin{align*}
    \Vert u(t) - u_{h,\Delta t}(t) \Vert_{L^p(\Omega ; H)} &\leq C_{p,\theta} e^{c \lambda t} (h^{\theta} + \Delta t^{\theta / 2}) (1 + t^{-\theta / 2 + \rho / 2} \Vert (\lambda + A_1)^{\rho / 2} \xi \Vert_{L^p(\Omega ; H)}),
\end{align*}
for any $\theta \in [0, 2\gamma + 1 - d/2) \cap [0,2]$, and $\rho \leq \theta$. 
\end{proof}

The final lemma is used to show that the strong convergence rate in Theorem \ref{theorem:strong-rate} gives the rate of pathwise convergence described in Corollary \ref{cor:pathwise-rate}.
\begin{lemma}\label{lemma:pathwise}
    Let $y_{\phi} : \Omega \to [0,\infty)$, $0 < \phi < 1$, be a collection of random variables, and suppose that for some $p > 0$ there is a constant $C > 0$ such that $E[y_{\phi}^p]^{1/p} \leq C \phi$. Then for any $\epsilon > 0$ small, and decreasing sequence $\phi_n \in (0,1)$ satisfying $\sum_{n = 1}^{\infty} \phi_n^{p \epsilon} < \infty$, there is $M_{\epsilon} : \Omega \to [0, \infty)$, ensuring that
    \begin{align*}
        y_{\phi_n} \leq M_{\epsilon} \phi_n^{1-\epsilon}, \quad P\text{-a.s.} 
    \end{align*}
\end{lemma}
\begin{proof}
    We will construct $M_{\epsilon}$ explicitly. For any $\epsilon > 0$, define $f_{\epsilon,n} := y_{\phi_n} / \phi_n^{1-\epsilon}$. Since,
    \begin{align*}
        P(f_{\epsilon,n} > 1) \leq E[f_{\epsilon,n}^p] \leq C \phi_n^{p \epsilon}, \quad \text{ and therefore } \quad  \sum_{n = 1}^{\infty} P(f_{\epsilon,n} > 1) < \infty,
    \end{align*}
    we have that $P(\limsup_{n \to \infty } \{ f_{\epsilon,n} > 1 \}) = 0$ by the Borel--Cantelli lemma. Therefore, $\sup_n f_{\epsilon,n} < \infty$ $P$-a.s., and defining
    \begin{align*}
        M_{\epsilon}(\omega) := 
        \begin{cases}
        \sup_{n \in \mathbb{N}} f_{n,\epsilon}(\omega), \quad &\text{when }\sup_{n \in \mathbb{N}} f_{n,\epsilon}(\omega) < \infty, \\
        0, \quad &\text{otherwise},
        \end{cases}
    \end{align*}
    we find $y_{\phi_n} = f_{\epsilon,n} \phi^{1-\epsilon} \leq M_{\epsilon} \phi^{1-\epsilon}$ $P\text{-a.s.}$
\end{proof}

\section{Numerical example and verification of convergence rate}

We numerically verify the convergence rates obtained in Theorem \ref{theorem:strong-rate} for the model,
\begin{equation}\label{ga:example:eq:SPDE}
    du = \Delta u \, dt +(I - \Delta)^{-\gamma} \, dW, \quad u(0) = 0,
\end{equation}
where $\Delta$ has zero Dirichlet boundary conditions. In Example \ref{ga:example1} and \ref{ga:example2}, we consider the domains $\mathcal{D} = (0,1)$ and $\mathcal{D} = (0,1)^2$, respectively. We let $H = L^2(\mathcal{D})$, and approximate the relative strong error at time $t = 1$ by
\begin{equation}\label{ga:example:eq:error}
    e_{h,\Delta t} := \frac{\Vert u_{h,\Delta t}(1)-u_{\tilde{h},\widetilde{\Delta t}}(1) \Vert_{L^2(\Omega;H)}}{\Vert u_{\tilde{h},\widetilde{\Delta t}}(1)\Vert_{L^2(\Omega;H)}},
\end{equation}
where a coarse approximation, $u_{h, \Delta t}(1)$, is compared to a reference solution, $u_{\tilde{h}, \widetilde{\Delta t}}(1)$, based on a finer spatial and temporal resolution, $\tilde{h}$ and $\widetilde{\Delta t}$, respectively. The quadrature resolution $k$ is fixed as described in each example below, and dependence on the quadrature resolution is suppressed in the notation.

To assess strong convergence, we proceed as follows. For the reference resolutions, there are $N_{\tilde{h}}$ nodal basis functions, denoted $\tilde{\varphi}_1, \dots, \tilde{\varphi}_{N_{\tilde{h}}}$, and 
$\widetilde{N} = \widetilde{\Delta t}^{-1}$ time steps, with $\tilde{t}_n := n / \widetilde{N}$, $n = 0, \ldots, \widetilde{N}$. We simulate the coefficients of the nodal basis functions of the increments of the projected Wiener process for each time step by
\begin{align*}
    \tilde{\delta}^n = M_{\tilde{h}}^{-1}
    \begin{pmatrix}
    (\pi_{\tilde{h}}(W(\tilde{t}_n)-W(\tilde{t}_{n-1})), \tilde{\varphi}_1)_H \\
    \vdots \\
    (\pi_{\tilde{h}}(W(\tilde{t}_n)-W(\tilde{t}_{n-1})), \tilde{\varphi}_{N_{\tilde{h}}})_H
    \end{pmatrix}
    \sim \mathcal{N}(0, \Delta t M_{\tilde{h}}^{-1}), \quad n = 1, \ldots, \widetilde{N},
\end{align*}
where $M_{\tilde{h}}$ is the mass matrix (see \eqref{eq:fem-matrices}). With the simulated increments $\pi_{\tilde{h}}(W(\tilde{t}_n)-W(\tilde{t}_{n-1})) = \sum_{j=1}^{N_{\tilde{h}}}\tilde{\delta}_i^n \tilde{\varphi}_i$, $n = 1, \dots, \widetilde{N}$, we can compute \eqref{eq:scheme-basis-coefficients-1} and \eqref{eq:scheme-basis-coefficients-2}---note in particular that $\varrho^n = \Delta t^{-1/2} M_{\tilde{h}}^{1/2}\tilde{\delta}^n$. 

For any coarser temporal grid, with $\{t_1, \ldots, t_N\} \subseteq \{\tilde{t}_1, \dots, \tilde{t}_{\widetilde{N}}\}$, increments are constructed by summing increments from the reference temporal grid. For any coarser spatial mesh, with $\{\varphi_1, \dots, \varphi_{N_h}\} \subseteq \mathrm{span}\{\tilde{\varphi}_1, \dots, \tilde{\varphi}_{N_{\tilde{h}}}\}$, we can express the coefficients of the Wiener process projected onto the coarser finite element space by $\delta^n = M_h^{-1}A M_{\tilde{h}} \tilde{\delta}^n$, where $A_{ij} = \varphi_i(\tilde{x}_j)$, with $\tilde{x}_j$ the vertex corresponding to $\tilde{\varphi}_j$, and $M_h$ is the mass matrix based on the coarser mesh.

Monte Carlo estimates of the errors are computed as follows. We simulate $R$ realizations of the reference and coarse solution, $u_{\tilde{h},\widetilde{\Delta t}}^{(r)}(1) = \sum_{i = 1}^{N_{\tilde{h}}} \tilde{\alpha}^{(r)}\tilde{\varphi}_i$ and $u_{h,\Delta t}^{(r)}(1) = \sum_{i = 1}^{N_h}\alpha^{(r)} \varphi_i$, respectively. The Monte Carlo estimate of the error, $\hat{e}_{h,\Delta t}$, in \eqref{ga:example:eq:error} is defined by
\begin{align}
    \hat{e}_{h,\Delta t}^2 := \frac{\frac{1}{R} \sum_{r = 1}^R(A^T\alpha^{(r)}-\tilde{\alpha}^{(r)})^T M_{\tilde{h}}(A^T\alpha^{(r)}-\tilde{\alpha}^{(r)})}{\frac{1}{R} \sum_{r = 1}^R (\tilde{\alpha}^{(r)})^T M_{\tilde{h}} \tilde{\alpha}^{(r)}}. \label{ga:example:eq:MonteCarloError}
\end{align}
In \eqref{ga:example:eq:MonteCarloError} the $H$-error is computed exactly, and so the only error in approximating \eqref{ga:example:eq:error} is due to the Monte Carlo error.

\begin{example}\label{ga:example1}
We consider \eqref{ga:example:eq:SPDE} with $\mathcal{D} = (0,1)$, $k = 0.5$, $\widetilde{\Delta t} = 2^{-22}$ and $\tilde{h} = 2^{-11}$. Coarser solutions are computed at spatial resolutions $h \in \{2^{-4}, \dots, 2^{-9}\}$ and temporal resolutions $\Delta t \in \{2^{-14}, \dots, 2^{-20}\}$. The Monte Carlo estimates of the errors in \eqref{ga:example:eq:MonteCarloError} are computed using $R = 50$ realizations. 

The first row in Figure \ref{ga:example1:fig:rates} shows the numerical convergence rates for $\gamma = 0$, $\gamma = 0.25$, $\gamma = 0.5$ and $\gamma = 0.75$, together with corresponding theoretical rates from Theorem \ref{theorem:strong-rate}. The figure shows good correspondence between numerical and theoretical rates.

\begin{figure}
    \centering
    \subcaptionbox{Dashed lines show rates $\frac{1}{2}, 1, \frac{3}{2}, 2$.}{
        \includegraphics[width = 6.5cm]{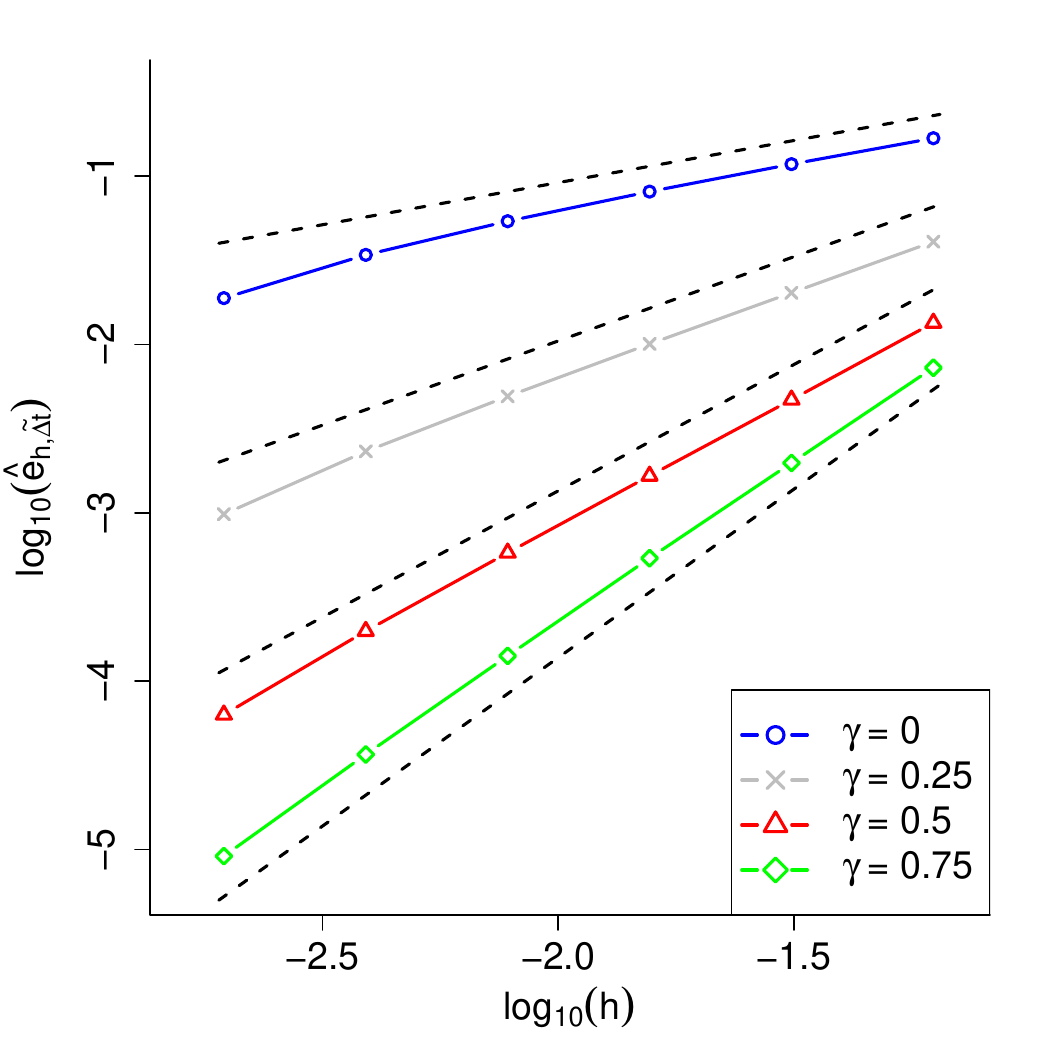}
    }
    \subcaptionbox{Dashed lines show rates $\frac{1}{4},\frac{1}{2},\frac{3}{4}, 1$}{
        \includegraphics[width = 6.5cm]{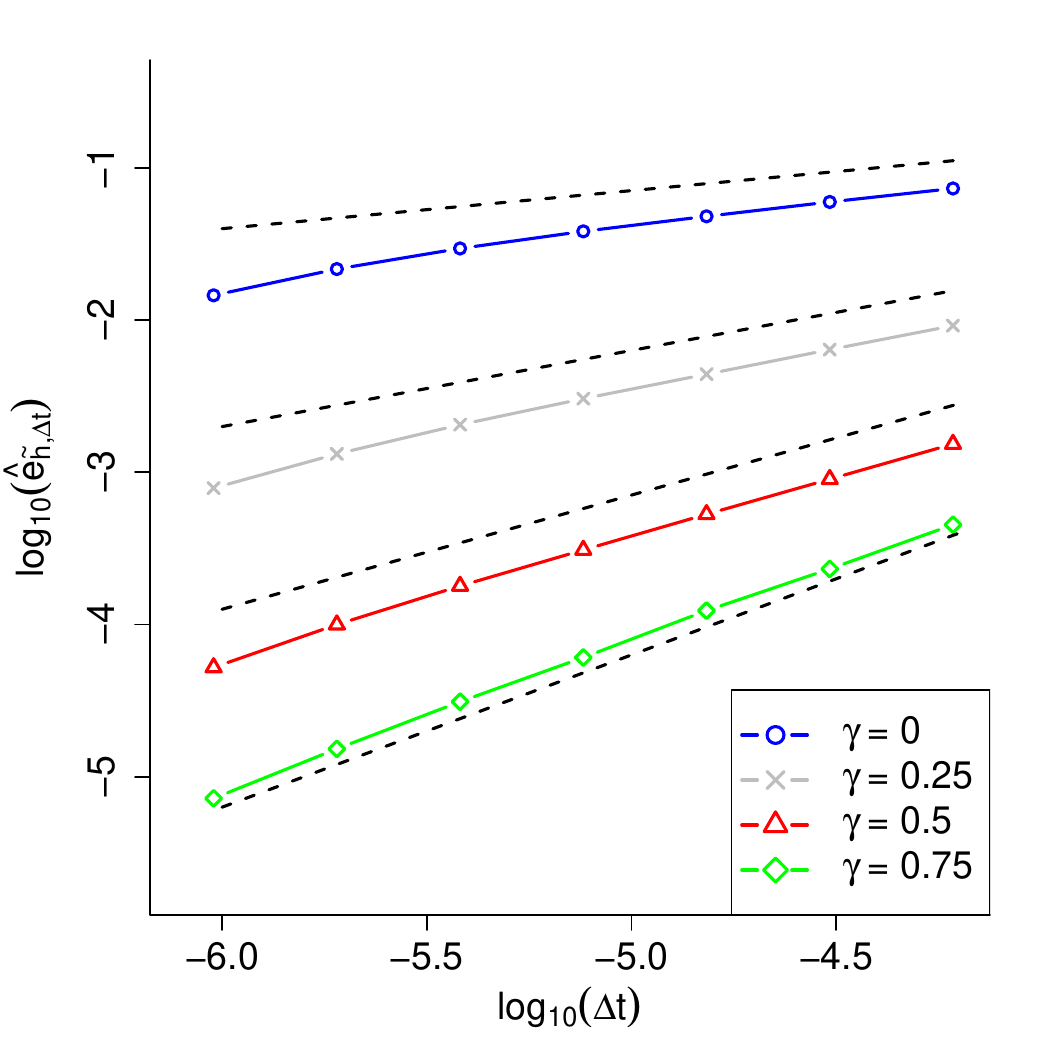}
    } \\
    \subcaptionbox{Dashed lines show rates $\frac{1}{2},1,\frac{3}{2},2$}{
        \includegraphics[width = 6.5cm]{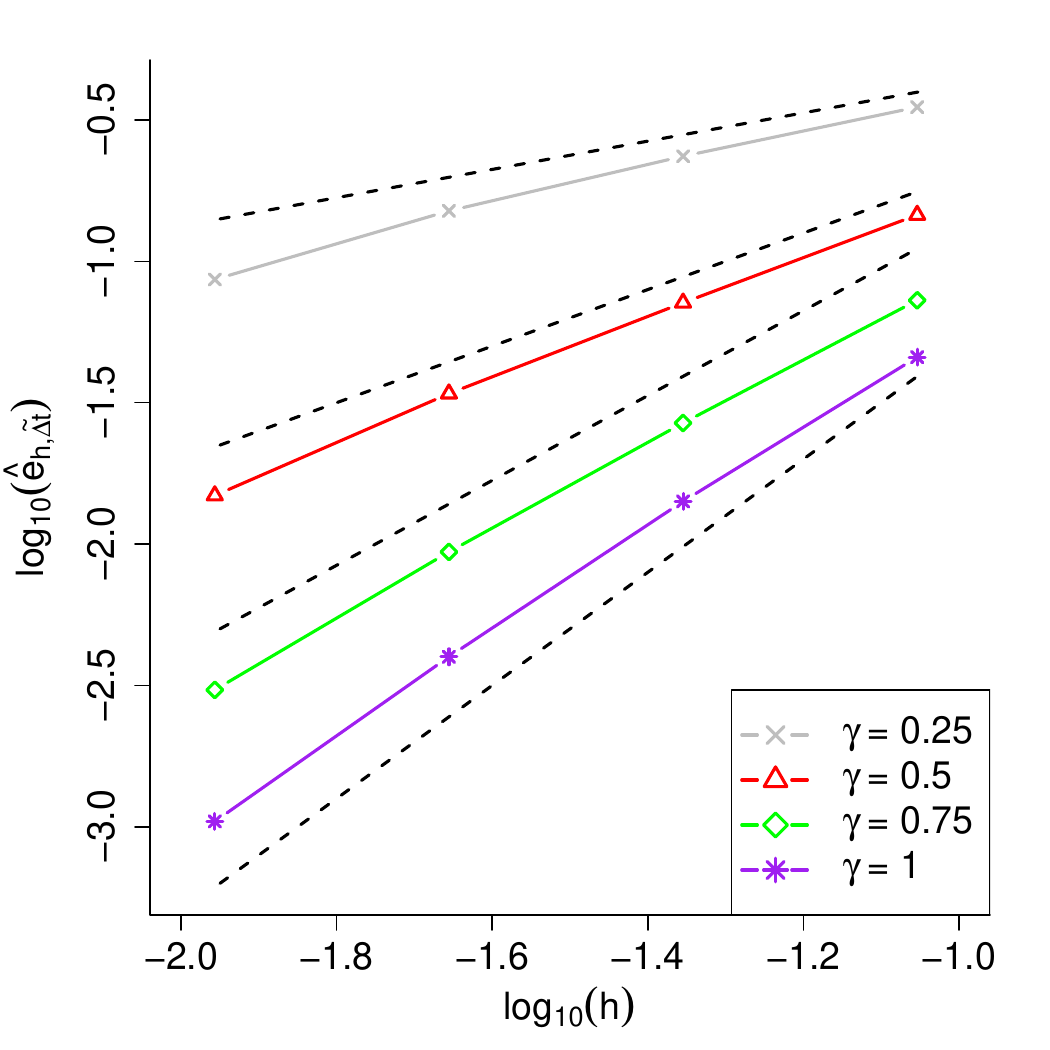}
    }
    \subcaptionbox{Dashed lines show rates $\frac{1}{4},\frac{1}{2},\frac{3}{4},1$.}{
        \includegraphics[width = 6.5cm]{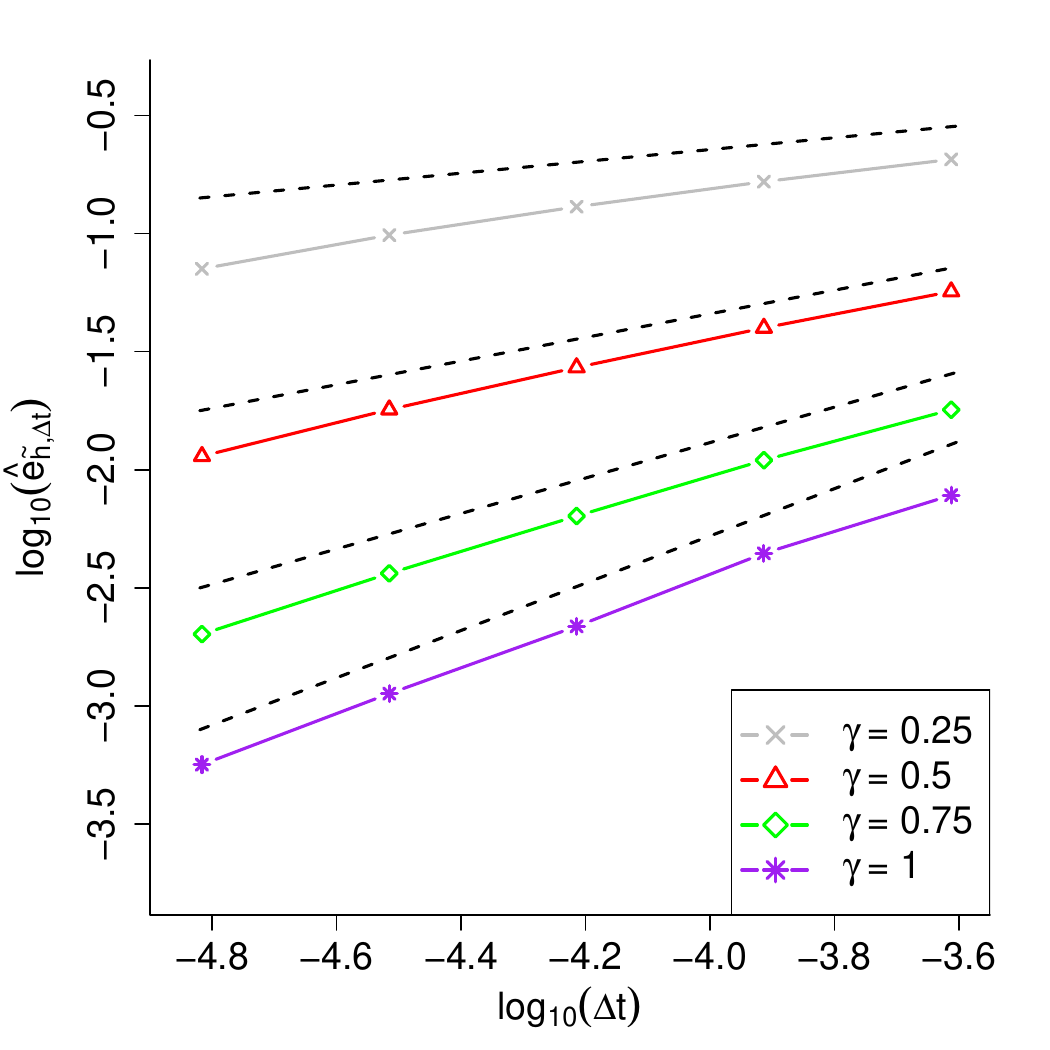}
    }
    \caption{Relative errors. Row 1 and 2 corresponds to Example 1 and 2, while Column 1 and 2 show rates in space and time, respectively. The dashed lines show corresponding theoretical asymptotic rates.\label{ga:example1:fig:rates}} 
\end{figure}

\end{example}

\begin{example}\label{ga:example2}
We consider \eqref{ga:example:eq:SPDE} with $\mathcal{D} = (0,1)^2$, $k = 0.5$, $\widetilde{\Delta t} = 2^{-19}$ and $\tilde{h} = 2^{-8.5}$. The mesh for the reference solution is created by dividing $\mathcal{D} = (0,1)^2$ into a uniform grid of $2^9 \times 2^9$ squares, and then dividing each square into two triangles with a diagonal line from the top-left corner to the bottom-right corner. Coarser solutions are computed at spatial resolutions $h \in \{2^{-3.5}, \dots, 2^{-6.5}\}$, and temporal resolutions $\Delta t \in \{2^{-12}, \dots, 2^{-16}\}$. The Monte Carlo estimates of the errors in \eqref{ga:example:eq:MonteCarloError} are computed using $R = 5$ realizations. We speed up computations by taking advantage of the fact that in this example $(1-\Delta)^{-\gamma}$ and $\Delta$ commute so that $(1-\Delta)^{-\gamma}$ only needs to be applied at the final time.

The second row in Figure \ref{ga:example1:fig:rates} shows the numerical convergence rates for $\gamma = 0.25$, $\gamma = 0.5$, $\gamma = 0.75$ and $\gamma = 1$, together with corresponding theoretical rates from Theorem \ref{theorem:strong-rate}. The figure demonstrates the decrease in rate of $0.5$ in space and $0.25$ in time compared to Example \ref{ga:example1} due to the dimension of the domain being increased by 1.
\end{example}

\begin{example}\label{ga:example3}
In our final example we consider a model with spatially varying coefficients. Stochastic advection-diffusion-reaction equations driven by Wiener processes with Whittle--Matérn spatial covariance are an active area of research in spatio-temporal statistics \cite{2024-lindgren,2024-clarotto,2024-berild}. Generalizing the SPDE in \cite{2024-berild} to allow for fractional power Whittle--Matérn noise, and inspired by the parametrization in \cite{2024-lindgren}, we consider
\begin{align}\label{ga:eq:NonStatSPDE}
    \beta_1\, du = -(\beta_2^2-\nabla\cdot (a + vv^\mathrm{T})\nabla+b\cdot\nabla ) u\,dt+\beta_3^{-1}(\beta_2^2-\nabla\cdot (a + vv^\mathrm{T})\nabla)^{-\gamma}\,dW, \quad u(0) = 0,
\end{align}
with domain $\mathcal{D} = (0,1)^2$, zero Dirichlet boundary conditions, and parameters
\begin{center}
    \begin{tabular}{ccccccc}
    $\gamma$ & $\beta_1$ & $\beta_2$ & $a$ & $v$ & $b$ & $\beta_3$ \\
    \hline
    0.25 & 10 & 8 & 1 & $\begin{pmatrix}-8(y-0.5)\\8(x-0.5)\end{pmatrix}$ & $ \begin{pmatrix}200\\0\end{pmatrix}$ & $2\cdot 10^{-2}$ \\[0.4cm]
    0.75 & 30 & 14 & 1 & $\begin{pmatrix}-8(y-0.5)\\8(x-0.5)\end{pmatrix}$ & $ \begin{pmatrix}200\\0\end{pmatrix}$ & $5.6\cdot 10^{-4}$
    \end{tabular}
\end{center}
where the aim is to create strong dependence in circles around the center $(0.5,0.5)^\mathrm{T}$ together with strong transport in the positive $x$-direction. The former is achieved through spatial anisotropy in the driving noise and anisotropic diffusion. The latter is achieved through advection. The other values were chosen to achieve reasonable correlation ranges in space and time, and marginal variances of order $1$.

Figure \ref{ga:example3:fig:NonStat} displays realizations of the solution at three different time points. The mesh is created as in Example \ref{ga:example2} with $h = 2^{-6.5}$, $\Delta t = 2^{-12}$, $k = 0.5$. For $\gamma = 0.25$, spatial regularity is $0.5$ resulting in a somewhat rough appearance in space, while for $\gamma = 0.75$, spatial regularity is $1.5$ resulting in a smoother appearance.
\end{example}

\subsection*{Acknowledgements}

The research of ØSA, GAF and ERJ was supported by the project IMod---Partial differential equations, statistics and data: an interdisciplinary approach to data-based modelling, project number 325114, from the Research Council of Norway. AL's work was supported in part by the Swedish Research Council (VR) through grant no.\ 2020-04170, by the Wallenberg AI, Autonomous Systems and Software Program (WASP) funded by the Knut and Alice Wallenberg Foundation, by the Chalmers AI Research Centre (CHAIR), and by the European Union (ERC, StochMan, 101088589). Views and opinions expressed are however those of the author(s) only and do not necessarily reflect those of the European Union or the European Research Council Executive Agency. Neither the European Union nor the granting authority can be held responsible for them.

\addcontentsline{toc}{section}{References}
\bibliographystyle{abbrv}
\bibliography{refs.bib}

\appendix
\input{appendix-A.tex}
\input{appendix-B.tex}

\end{document}

%% file: appendix-A.tex
\section{Proofs of mild solution existence, and space and time regularity}\label{app:prelim-proofs}

\begin{proof}[Proof of Lemma \ref{lemma:mild-solution-existence}]
$S(t - \cdot) A_2^{-\gamma}$ is an admissible integrand provided, 
\begin{align}\label{eq:admissible-integrand-condition}
    \int_0^t \Vert S_1(t - s) A_2^{-\gamma} \Vert_{L_2(H)}^2 \, ds < \infty.
\end{align}
To show that \eqref{eq:admissible-integrand-condition} holds for any $\gamma > d/4 - 1/2$, we first note that owing to condition \ref{cond:A1A2} and Lemma \ref{lemma:analytic-semigroup} (a), we have
\begin{align*}
    \Vert S_1(t) A_2^{1/2} \Vert_{L(H)} &= \Vert (\lambda + A_1)^{1/2} S_1(t) (\lambda + A_1)^{-1/2} A_2^{1/2} \Vert_{L(H)} \\
    &\leq \Vert (\lambda + A_1)^{1/2} S_1(t) \Vert_{L(H)} \Vert (\lambda + A_1)^{-1/2} A_2^{1/2} \Vert_{L(H)} \\
    &\leq C e^{(\lambda - \delta) t} t^{-1/2}.
\end{align*}
By using Lemma 2.5 in \cite{2024-auestad}, we find that $\Vert S_1(t) A_2^{\alpha} \Vert_{L(H)} \leq C e^{(\lambda - \delta) t} t^{-\alpha}$, for $\alpha \in [0,1/2]$. Therefore, for any $0 < \epsilon < \gamma + 1 / 2 - d / 4$
\begin{align*}
    \int_0^t \Vert S_1(t-s) A_2^{-\gamma} \Vert_{L_2(H)}^2 \, ds &= \int_0^t \Vert S_1(t-s) A_2^{1/2 - \epsilon} A_2^{-1/2 + \epsilon - \gamma} \Vert_{L_2(H)}^2 \, ds \\
    &\leq \int_0^t \Vert S_1(t-s) A_2^{1/2 - \epsilon} \Vert_{L(H)}^2 \Vert A_2^{-1/2 + \epsilon - \gamma} \Vert_{L_2(H)}^2 \, ds \\
    &\leq C e^{2\lambda t} \int_0^t e^{-2\delta s} s^{-1 + 2\epsilon} \Vert A_2^{-1/2 + \epsilon - \gamma} \Vert_{L_2(H)}^2 \, ds \\
    &\leq C_{\epsilon} e^{2\lambda t} \Vert A_2^{-1/2 + \epsilon - \gamma} \Vert_{L_2(H)}^2 < \infty, 
\end{align*}
by condition \ref{cond:hilbert-schmidt}. 

The bound on $\Vert u(t) \Vert_{L^p(\Omega ; H)}$ follows by combining the observations above with the BDG inequality (Theorem 4.36 in \cite{da-prato}), and noting that by Lemma \ref{lemma:analytic-semigroup}, 
\begin{align*}
    \Vert S_1(t) \xi \Vert_{L^p(\Omega ; H)} &\leq E[\Vert (\lambda + A_1)^{-\rho} S_1(t) \Vert_{L(H)}^p \Vert (\lambda + A_1)^{\rho} \xi \Vert_H^p]^{1/p} \\ 
    &\leq C e^{(\lambda - \delta) t} t^{\rho} \Vert (\lambda + A_1)^{\rho} \xi \Vert_{L^p(\Omega ; H)},
\end{align*}
for any $\rho \leq 0$. 
\end{proof}

\begin{proof}[Proof of Lemma \ref{lemma:space-regularity-mild-solution}]
    Using that $(\lambda + A_1)^{\alpha}$ is closed, we find 
    \begin{align*}
        (\lambda + A_1)^{\alpha} u(t) = (\lambda + A_1)^{\alpha} S_1(t) \xi + \int_0^t (\lambda + A_1)^{\alpha} S_1(t-s) A_2^{-\gamma} \, dW.
    \end{align*}
    For the first term above, we have by Lemma \ref{lemma:analytic-semigroup},
    \begin{align*}
        \Vert (\lambda + A_1)^{\alpha} S_1(t) \xi \Vert_{L^p(\Omega ; H)} &\leq E[\Vert (\lambda + A_1)^{\alpha-\rho} S_1(t) \Vert_{L(H)}^p \Vert (\lambda + A_1)^{\rho} \xi \Vert_H^p]^{1/p} \\
        &\leq C e^{(\lambda - \delta) t} t^{-\alpha + \rho} \Vert (\lambda + A_1)^{\rho} \xi \Vert_{L^p(\Omega ; H)}. 
    \end{align*}

    For the next term, we have using the BDG inequality (Theorem 4.36 in \cite{da-prato}),
    \begin{align*}
        \Vert \int_0^t (\lambda + A_1)^{\alpha} S_1(t-s) A_2^{-\gamma} \, dW \Vert_{L^p(\Omega ; H)}^p &\leq C_p \bigg( \int_0^t \Vert (\lambda + A_1)^{\alpha} S_1(t-s) A_2^{-\gamma} \Vert_{L_2(H)}^2 \, ds \bigg)^{p / 2}.
    \end{align*}
    Let $\epsilon > 0$ be small, and note that
    \begin{align*}
        \Vert (\lambda + A_1)^{\alpha} S_1(t-s) A_2^{-\gamma} \Vert_{L_2(H)}^2 &= \Vert (\lambda + A_1)^{\alpha} S_1(t-s) A_2^{-\gamma + d/4 + \epsilon} A_2^{-d/4 - \epsilon} \Vert_{L_2(H)}^2 \\
        &\leq \Vert (\lambda + A_1)^{\alpha} S_1(t-s) A_2^{-\gamma + d/4 + \epsilon} \Vert_{L(H)}^2 \Vert A_2^{-d/4 + \epsilon} \Vert_{L_2(H)}^2,
    \end{align*}
    with the second factor being finite by \ref{cond:hilbert-schmidt}. For the first factor, note that for any $\alpha \geq 0$ we have by Lemma \ref{lemma:analytic-semigroup} and \ref{cond:A1A2}
    \begin{align*}
        \Vert (\lambda + A_1)^{\alpha} S_1(t) \Vert_{L(H)} &\leq C e^{(\lambda - \delta) t} t^{-\alpha}, \\
        \Vert (\lambda + A_1)^{\alpha} S_1(t) A_2^{-1} \Vert_{L(H)} &\leq \Vert (\lambda + A_1)^{\alpha - 1} S_1(t) \Vert_{L(H)} \Vert (\lambda + A_1) A_2^{-1} \Vert_{L(H)} \\
        &\leq C e^{(\lambda - \delta) t} t^{\min(0, -\alpha + 1)}, \\
        \Vert (\lambda + A_1)^{\alpha} S_1(t) A_2^{1/2} \Vert_{L(H)} &\leq \Vert (\lambda + A_1)^{\alpha + 1/2} S_1(t) \Vert_{L(H)} \Vert (\lambda + A_1)^{-1/2} A_2^{1/2} \Vert_{L(H)} \\ &\leq C e^{(\lambda - \delta) t} t^{-\alpha - 1/2}.
    \end{align*}
    Therefore, by interpolation (Lemma 2.5 in \cite{2024-auestad}), we find that for any $\beta \in (-\infty, 1/2]$
    \begin{align*}
        \Vert (\lambda + A_1)^{\alpha} S_1(t) A_2^{\beta} \Vert_{L(H)} \leq C e^{(\lambda - \delta) t} t^{\min(0, -\alpha - \max(-1, \beta))},
    \end{align*}
    where we used that negative fractional powers of $(\lambda + A_1)$ and $A_2$ are bounded. 
    
    Thus, for the first factor, we have
    \begin{align*}
        \Vert (\lambda + A_1)^{\alpha} S_1(t-s) A_2^{-\gamma + d/4 + \epsilon} \Vert_{L(H)}^2 \leq C e^{2(\lambda - \delta) t} t^{2\min(0, -\alpha - \max(-1, -\gamma + d / 4 + \epsilon))},
    \end{align*}
    and therefore need $-\alpha - \max(-1, -\gamma + d/4 + \epsilon) > -1/2$, which holds if $\alpha < \min(3 / 2, 1 / 2 + \gamma - d / 4 - \epsilon)$. By choosing $\epsilon$ smaller and depending on $\alpha$, we find that the integrand is integrable for any $\alpha < \min(3/2, 1 / 2 + \gamma - d / 4)$.
\end{proof}

\begin{proof}[Proof of Lemma \ref{lemma:time-regularity-mild-solution}]
    We decompose the difference as follows,
    \begin{align*}
        u(t_2) - u(t_1) &= (S_1(t_2) - S_1(t_1)) \xi + \bigg( \int_0^{t_2} S_1(t_2 - s) A_2^{-\gamma} \, dW - \int_0^{t_1} S_1(t_1 - s) A_2^{-\gamma} \, dW \bigg).
    \end{align*}
    For the first term, we have by Lemma \ref{lemma:analytic-semigroup}
    \begin{align*}
        &\Vert (S_1(t_2) - S_1(t_1)) \xi \Vert_{L^p(\Omega ; H)} = E[\Vert (S_1(t_2) - S_1(t_1)) \xi \Vert_H^p ]^{1/p} \\
        &\qquad = E[\Vert (S_1(t_2 - t_1) - I) (\lambda + A_1)^{-\alpha} (\lambda + A_1)^{\alpha - \rho} S_1(t_1) (\lambda + A_1)^{\rho} \xi \Vert_H^p]^{1/p} \\
        &\qquad \leq E[\Vert (S_1(t_2 - t_1) - I) (\lambda + A_1)^{-\alpha} \Vert_{L(H)}^p \Vert (\lambda + A_1)^{\alpha - \rho} S_1(t_1) \Vert_{L(H)}^p \Vert (\lambda + A_1)^{\rho} \xi \Vert_H^p]^{1/p} \\
        &\qquad \leq C e^{\lambda t_2 - \delta t_1} (t_2 - t_1)^{\alpha} t^{-\alpha + \rho} E[\Vert (\lambda + A_1)^{\rho} \xi \Vert_H^p]^{1/p},
    \end{align*}
    for some $\alpha \in [0,1]$ and $\rho \leq \alpha$. 

    For the second term, we have,
    \begin{align*}
        \int_0^{t_2} S_1(t_2 - s) A_2^{-\gamma} \, dW - \int_0^{t_1} S_1(t_1 - s) A_2^{-\gamma} \, dW &= \int_0^{t_1} (S_1(t_2-s) - S_1(t_1-s)) A_2^{-\gamma} \, dW \\
        &\qquad+ \int_{t_1}^{t_2} S_1(t_2 - s) A_2^{-\gamma} \, dW \\
        &= (i) + (ii).
    \end{align*}
    For $(i)$ we have using the BDG inequality (Theorem 4.36 in \cite{da-prato}),
    \begin{align*}
        &\Vert \int_0^{t_1} (S_1(t_2-s) - S_1(t_1-s)) A_2^{-\gamma} \, dW \Vert_{L^p(\Omega ; H)}^p \\
        &\qquad \leq C_p \bigg( \int_0^{t_1} \Vert (S_1(t_2-s) - S_1(t_1-s)) A_2^{-\gamma} \Vert_{L_2(H)}^2 \, ds \bigg)^{p / 2}.
    \end{align*}
    For the integrand, we have by condition \ref{cond:hilbert-schmidt} that for some $\epsilon > 0$ small,
    \begin{align*}
        &\Vert (S_1(t_2-s) - S_1(t_1-s)) A_2^{-\gamma} \Vert_{L_2(H)}^2 \\
        &\qquad = \Vert (S_1(t_2-t_1) - I) S_1(t_1 - s) A_2^{-\gamma + d/4 + \epsilon} A_2^{-d/4 - \epsilon} \Vert_{L_2(H)}^2 \\
        &\qquad \leq \Vert (S_1(t_2-t_1) - I) S_1(t_1 - s) A_2^{-\gamma + d/4 + \epsilon} \Vert_{L(H)}^2 \Vert A_2^{-d/4 - \epsilon} \Vert_{L_2(H)}^2 \\
        &\qquad \leq C_{\epsilon} \Vert (S_1(t_2-t_1) - I) S_1(t_1 - s) A_2^{-\gamma + d/4 + \epsilon} \Vert_{L(H)}^2.
    \end{align*}
    By Lemma \ref{lemma:analytic-semigroup} we note that for $t, \Delta t > 0$,
    \begin{align*}
        \Vert (S_1(\Delta t) - I) S_1(t) (\lambda + A_1)^{\alpha} \Vert_{L(H)} \leq C e^{(\lambda - \delta) t + \lambda \Delta t} t^{\min(0, -\beta - \alpha)} \Delta t^{\beta}, 
    \end{align*}
    for any $\alpha$, and $\beta \in [0,1]$. In the case of $\alpha \in \{ -1, 1/2 \}$ we can by condition \ref{cond:A1A2} replace the power of $\lambda + A_1$ by the same power of $A_2$ (by similar arguments as in the proof of Lemma \ref{lemma:space-regularity-mild-solution}), and by interpolation (Lemma 2.5 in \cite{2024-auestad}), we see that for $\alpha \leq 1/2$ and $\beta \in [0,1]$
    \begin{align*}
        \Vert (S_1(\Delta t) - I) S_1(t) A_2^{\alpha} \Vert_{L(H)} \leq C e^{(\lambda - \delta) t + \lambda \Delta t} t^{\min(0, -\beta - \max(-1,\alpha))} \Delta t^{\beta}.
    \end{align*}
    Therefore, setting $\Delta t = t_2 - t_1, \ t = t_1 - s$ and $\beta = \min(1, 1 / 2 - \epsilon - \max(-1,-\gamma + d / 4 + \epsilon)) = \min(1, \gamma + 1 / 2 - d / 4 - 2 \epsilon)$) above, we find
    \begin{align*}
        &\Vert (S_1(t_2-t_1) - I) S_1(t_1 - s) A_2^{-\gamma + d/4 + \epsilon} \Vert_{L(H)} \\
        &\qquad \leq C e^{\lambda (t_2 - s) - \delta (t_1 - s)} (t_1 - s)^{-1/2 + \epsilon} (t_2 - t_1)^{\min(1, \gamma + 1/2 - d/4 - 2\epsilon)},
    \end{align*}
    where we require $\epsilon < (\gamma + 1/2 - d/4) / 2$. Thus
    \begin{align*}
        \Vert (i) \Vert_{L^p(\Omega; H)} \leq C_{\epsilon} e^{\lambda t_2} (t_2 - t_1)^{\min(1, \gamma + 1/2 - d/4 - 2\epsilon)}.
    \end{align*}
    
    For $(ii)$ we have by the BDG inequality (Theorem 4.36 in \cite{da-prato}),
    \begin{align*}
        \Vert \int_{t_1}^{t_2} S_1(t_2 - s) A_2^{-\gamma} \, dW \Vert_{L^p(\Omega ; H)}^p \leq C_p \bigg( \int_{t_1}^{t_2} \Vert S_1(t_2 - s) A_2^{-\gamma} \Vert_{L_2(H)}^2 \, ds \bigg)^{p / 2}.
    \end{align*}
    For the integrand, we have by \ref{cond:hilbert-schmidt} and similar arguments as in the proof of Lemma \ref{lemma:mild-solution-existence}, that for some $0 < \epsilon < \gamma + 1 / 2 - d / 4$ small
    \begin{align*}
        &\Vert S_1(t_2 - s) A_2^{-\gamma + d / 4 + \epsilon} A_2^{-d/4 - \epsilon} \Vert_{L_2(H)}^2 \\
        &\qquad \leq \Vert S_1(t_2 - s) A_2^{-\gamma + d / 4 + \epsilon} \Vert_{L(H)}^2 \Vert A_2^{-d/4 - \epsilon} \Vert_{L_2(H)}^2 \\
        &\qquad \leq C_{\epsilon} e^{2(\lambda - \delta) (t_2 - s)}
        \begin{cases}
            (t_2 - s)^{2\gamma - d / 2 - 2\epsilon}, \quad &-\gamma + d / 4 + \epsilon > 0, \\
            1, \quad &\text{otherwise}. 
        \end{cases}
    \end{align*}
    By evaluating the integral, we therefore, have,
    \begin{align*}
        \Vert (ii) \Vert_{L^p(\Omega ; H)} \leq C_{\epsilon, p} e^{\lambda t_2}
        \begin{cases}
            (t_2 - t_1)^{\gamma + 1/2 - d / 4 - \epsilon}, \quad &-\gamma + d/4 + \epsilon > 0,\\
            (t_2 - t_1)^{1/2}, \quad &\text{otherwise}.
        \end{cases}
    \end{align*}
    
    Combining all of these inequalities finishes the proof of the inequality. The Hölder continuity follows by Kolmogorov continuity test (see e.g. Theorem 3.3 in \cite{da-prato}).
\end{proof}

%% file: appendix-B.tex
\section{Derivation of (3.4) and (3.5)}\label{app:derivation-of-system}

To rewrite \eqref{eq:fully-discrete-scheme} as a system of equations for the coefficients of $u_{h,\Delta t}$ in the nodal basis, the following lemma is key. 
\begin{lemma}\label{lemma:discrete-cylindrical-wiener-process}
    Let $W$ be a cylindrical Wiener process on $H$. Then $\pi_h W$ is a cylindrical Wiener process on $V_h$, in the sense that,
    \begin{align*}
        \pi_h W (t) = \sum_{j = 1}^{N_h} \beta_j(t) e_{h,j}, \quad P\text{-a.s.}
    \end{align*}
    where $\beta_j$ are independent scalar Brownian motions, and $\{ e_{j,h} \}_{j = 1}^{N_h}$ is an $H$-orthonormal basis of $V_h$. 
\end{lemma}
\begin{proof}
    We have that $\pi_h \in L_2(H, V_h)$ (where $V_h$ has the $H$-norm) since it has finite rank. It therefore follows that $\pi_h W(t) \in V_h$ $P$-a.s. (see e.g. Section 4.2.1 in \cite{da-prato}). Therefore, for any $H$-orthonormal basis $\{ e_{h,j} \}_{j = 1}^{N_h}$ of $V_h$ we must have
    \begin{align*}
        \pi_h W(t) = \sum_{j = 1}^{N_h} (\pi_h W(t), e_{h,j})_H e_{h,j}, \quad P\text{-a.s.}
    \end{align*}
    Moreover, $(\pi_h W(t), e_{h,j})_H$ are independent scalar Brownian motions for each $j$, since $(\pi_h W(t), e_{h,j})_H = (W(t), e_{h,j})_H$.
\end{proof}

In order to derive \eqref{eq:scheme-basis-coefficients-1} and \eqref{eq:scheme-basis-coefficients-2} we first note that solving \eqref{eq:fully-discrete-scheme} is the same as solving the system of equations, 
\begin{align}\label{eq:system-of-equations-appendix}
    ((I + \Delta t A_{1,h})u_{h,\Delta t}(t_{n+1}) , \varphi_j)_H = ( u_{h,\Delta t}(t_n) + Q_k^{-\gamma}(A_{2,h}) \pi_h (W(t_{n+1}) - W(t_n)), \varphi_j)_H,
\end{align}
$j = 1, \dots, N_h$, where $\{\varphi_j\}_{j = 1}^{N_h}$ is the nodal basis of $V_h$. To set up this system, the following lemma will be useful.
\begin{lemma}\label{lemma:auxiliary-basis-coef}
    For $\gamma = 1$, we have,
    \begin{align*}
        A_{2,h}^{-1} \pi_h (W(t_{n+1}) - W(t_n)) = \sum_{j=1}^{N_h} \theta_j \varphi_j, \quad P\text{-a.s.}
    \end{align*}
    where $\theta = \Delta t^{1/2} K_h^{-1} M_h^{1/2} \varrho^n$. Here, the matrices $M_h$ and $K_h$ are defined in \eqref{eq:fem-matrices}, $(\cdot)_j$ denotes the $j$'th entry of a vector,
    \begin{align*}
        \varrho^n := \Delta t^{-1/2} M_h^{-1 / 2}
        \begin{pmatrix}
            &(\pi_h (W(t_{n+1}) - W(t_n)), \varphi_1)_H \\
            &\vdots \\
            &(\pi_h (W(t_{n+1}) - W(t_n)), \varphi_{N_h})_H
        \end{pmatrix}
        \sim \mathcal{N}(0,I)
    \end{align*}
    are $N_h$-dimensional multivariate Gaussian. If $\gamma \in (0,1)$, we have,
    \begin{align*}
        Q_k^{-\gamma}(A_{2,h}) \pi_h (W(t_{n+1}) - W(t_n)) = \sum_{j=1}^{N_h} \theta_j' \varphi_j, \quad P\text{-a.s.}
    \end{align*}
    where $\theta' = \Delta t^{1/2} \frac{k \sin(\pi \gamma)}{\pi} \sum_{l = -M}^N e^{(1-\gamma) y_l} (e^{y_l} M_h + K_h)^{-1} M_h^{1/2} \varrho^n$. 
\end{lemma}
\begin{proof}
    Set for ease of notation,
    \begin{align*}
        f := \pi_h (W(t_{n+1}) - W(t_n)).
    \end{align*}
    By Lemma \ref{lemma:discrete-cylindrical-wiener-process}, $f$ is an $H$-valued Gaussian random variable with covariance operator $\Delta t \pi_h$. 

    For the first identity of the lemma, note that $A_{2,h}^{-1} f$ is the solution $g \in V_h$ of
    \begin{align*}
        A_{2,h} g = f.
    \end{align*}
    Since $f \in V_h$ $P$-a.s. by Lemma \ref{lemma:discrete-cylindrical-wiener-process}, solving the equation above is the same as solving the system of equations,
    \begin{align}\label{eq:system-1}
        \begin{pmatrix}
            (A_{2,h} g, \varphi_1)_H \\
            \vdots \\
            (A_{2,h} g, \varphi_{N_h})_H
        \end{pmatrix}
        =
        \begin{pmatrix}
            (f, \varphi_1)_H \\
            \vdots \\
            (f, \varphi_{N_h})_H
        \end{pmatrix}.
    \end{align}
    By Lemma \ref{lemma:discrete-cylindrical-wiener-process} we have,
    \begin{align*}
        E[(f, \varphi_i)_H (f, \varphi_j)_H] = \Delta t (\varphi_i, \varphi_j)_H,
    \end{align*}
    and so the covariance matrix of $((f, \varphi_1)_H, \dots, (f, \varphi_{N_h}))^T$ is the (scaled) mass matrix, $\Delta t M_h$. It follows that 
    \begin{align*}
        ((f, \varphi_1)_H, \dots, (f, \varphi_{N_h}))^T = \Delta t^{1/2} M_h^{1/2} \varrho^n,
    \end{align*}
    $P$-a.s., with $\varrho^n$ as above. For the left hand side of \eqref{eq:system-1} we insert $g = \sum_{j = 1}^{N_h} \theta_j \varphi_j$, and find, using that $(A_{2,h} \varphi_j, \varphi_i)_H = a_2(\varphi_j, \varphi_i)$,
    \begin{align*}
        \begin{pmatrix}
            (A_{2,h} g, \varphi_1)_H \\
            \vdots \\
            (A_{2,h} g, \varphi_{N_h})_H
        \end{pmatrix}
        = K_h 
        \begin{pmatrix}
            \theta_1 \\
            \vdots \\
            \theta_{N_h}
        \end{pmatrix},
    \end{align*}
    where $K_h$ is as in \eqref{eq:fem-matrices}. Combining these observations, we see that \eqref{eq:system-1} has solution $(g_1, \dots, g_{N_h})$ given by,
    \begin{align*}
        (g_1, \dots, g_{N_h})^T = \Delta t^{1/2} K_h^{-1} M_h^{1/2} \varrho^n.
    \end{align*}

    For the second identity of the lemma we argue similarly. Note that,
    \begin{align*}
        Q_k^{-\gamma}(A_{2,h}) f = \frac{k \sin(\pi \gamma)}{\pi} \sum_{j = -M}^N e^{(1-\gamma) y_j} g^{(j)}, 
    \end{align*}
    where $g^{(j)} \in V_h$, is the solution of the equation, 
    \begin{align*}
        (e^{y_j} I + A_{2,h}) g^{(j)} = f, \quad j = -M, \dots, N.
    \end{align*}
    As for the previous term, to solve this equation we insert $g^{(j)} = \sum_{l=1}^{N_h} \theta^{(j)}_l \varphi_l$ into the equation and integrate against the nodal basis, to find,
    \begin{align*}
        (e^{y_j} M_h + K_h) 
        \begin{pmatrix}
            \theta_1^{(j)} \\
            \vdots \\
            \theta_{N_h}^{(j)}
        \end{pmatrix}
        = \Delta t^{1/2} M_h^{1/2} \varrho^n.
    \end{align*}
    Summing up the vectors $(\theta^{(j)}_1, \dots, \theta^{(j)}_{N_h}), \ j = -M, \dots, N$ we find the coeffcients, $\theta'$, of $Q_k^{-\gamma} f$ in the nodal basis. This gives the second identity of the lemma. 
\end{proof}

Now we can insert the identities of Lemma \ref{lemma:auxiliary-basis-coef} into the system of equations \eqref{eq:system-of-equations-appendix} to arrive at \eqref{eq:scheme-basis-coefficients-1} and \eqref{eq:scheme-basis-coefficients-2}. Note that for any $g = \sum_{j=1}^{N_h} \theta_j \varphi_j$, we have,
\begin{align*}
    \begin{pmatrix}
        (g, \varphi_1)_H \\
        \vdots \\
        (g, \varphi_{N_h})_H
    \end{pmatrix}
    = M_h \theta,
\end{align*} 
and so,
\begin{align*}
    \begin{pmatrix}
        ((I + \Delta t A_{1,h})u_{h,\Delta t}(t_{n+1}), \varphi_1)_H \\
        \vdots \\
        ((I + \Delta t A_{1,h})u_{h,\Delta t}(t_{n+1}), \varphi_{N_h})_H
    \end{pmatrix}
    = (M_h + \Delta t T_{h}) \alpha^{n+1},
\end{align*} 
while, 
\begin{align*}
    &
    \begin{pmatrix}
        (Q_k^{-\gamma}(A_{2,h}) \pi_h (W(t_{n+1}) - W(t_n)), \varphi_1)_H \\
        \vdots \\
        (Q_k^{-\gamma}(A_{2,h}) \pi_h (W(t_{n+1}) - W(t_n)), \varphi_{N_h})_H
    \end{pmatrix}
    \\
    &\qquad= 
    \begin{cases}
        M_h \Delta t^{1/2} \frac{k \sin(\pi \gamma)}{\pi} \sum_{j = -M}^N e^{(1-\gamma) y_j} (e^{y_j} M_h + K_h)^{-1} M_h^{1/2} \varrho^n, \quad &\gamma \in (0,1), \\
        M_h \Delta t^{1/2} K_{h}^{-1} M_h^{1/2} \varrho^n, \quad &\gamma = 1.
    \end{cases}
\end{align*}